\documentclass{amsart}

\usepackage{graphicx,amsmath,amsthm,verbatim,amssymb,lineno,titletoc}
\usepackage{mathrsfs}
\usepackage{adjustbox}
\usepackage{ytableau}
\usepackage{bbm}
\usepackage{hyperref}
\usepackage{theoremref}
\usepackage[numbers]{natbib}
\usepackage{array}
\usepackage{tikz}
\usetikzlibrary{shapes.multipart,patterns,arrows}

 \usepackage{xcolor}
 \hypersetup{
 	colorlinks,
 	linkcolor={red!80!black},
 	citecolor={blue!80!black},
 	urlcolor={blue!80!black}
 }

\usepackage{amssymb}
\usepackage[all]{xy}
\usepackage{enumerate}
\usepackage{tikz}
\usepackage{mathrsfs}
\usepackage[english]{babel}
\usepackage{multirow}
\usepackage{float}
\usepackage{enumitem}
\usepackage{graphicx,accents}

\usepackage{enumitem}

\setlist[description]{leftmargin=1.1cm,labelindent=0.3cm}

\usepackage[font=small]{caption}

\usepackage{amssymb}

\usepackage{amsmath,amsfonts,amsthm,bbm} % Math packages

\usepackage{amssymb}
\usepackage[all]{xy}
\usepackage{enumerate}
\usepackage{tikz}
\usepackage{mathrsfs}
\usepackage[english]{babel}
\usepackage{multirow}
\usepackage{float}

\usepackage{graphicx,accents}

\DeclareMathOperator{\Var}{\textnormal{Var}}

\def\limsup{\mathop{\rm lim\,sup}\limits}

\def\Z{\mathbb{Z}}

\def\E{\mathcal{E}}

\def\Thom{\mathbb{T}^{\text{hom}}_d}

\newcommand{\Aut}{\textup{Aut}}

\newcommand{\ind}[1]{\mathbf{1}{\{ #1 \}}}
\newcommand{\f}{\frac}
\renewcommand{\Z}{\mathbb Z}
\renewcommand{\E}{\mathbf E}
\renewcommand{\P}{\mathbf P}
\newcommand{\B}{\mathbb B}
\newcommand{\Zo}{\vec{\mathbb Z}^d}
\DeclareMathOperator{\Poi}{Poi}

\newcommand{\HOX}[1]{\marginpar{\footnotesize #1}}

\newcommand{\Vb}{V^{\text{branch}}}
\renewcommand{\root}{\mathbf{0}}

\newtheorem{thm}{Theorem}
\numberwithin{thm}{section}
\newtheorem{lemma}[thm]{Lemma}
\newtheorem{prop}[thm]{Proposition}

\newtheorem{question}{Open Question}

\theoremstyle{remark}

\theoremstyle{definition}

\newtheorem{example}[thm]{Example}

\allowdisplaybreaks

\begin{document}
	
	\title{Parking on transitive unimodular graphs}

	\author[M.~Damron]{Michael Damron}
	\address{Michael Damron, Department of Mathematics, Georgia Institute of Technology, 
		Atlanta, GA 30332}
	\email{\texttt{mdamron6@gatech.edu}}

	\author[J.~Gravner]{Janko Gravner}
	\address{Janko Gravner, Department of Mathematics, University of California, 
		Davis, CA 95616}
	\email{\texttt{gravner@math.ucdavis.edu}}
	
	\author[M.~Junge]{Matthew Junge}
	\address{Matthew Junge, Department of Mathematics, Duke University, 
		Durham, NC 27708}
	\email{\texttt{jungem@math.duke.edu}}
	
		\author[H.~Lyu]{Hanbaek Lyu}
	\address{Hanbaek Lyu, Department of Mathematics, University of California, Los Angeles, CA 90095}
	\email{\texttt{colourgraph@gmail.com}}
	
	\author[D.~Sivakoff]{David Sivakoff}
	\address{David Sivakoff, Departments of Statistics and Mathematics, The Ohio State University, Columbus, OH 43210}
	\email{\texttt{dsivakoff@stat.osu.edu}}


	\begin{abstract}
		 Place a car independently with probability $p$ at each site of a graph. Each initially vacant site is a parking spot that can fit one car. Cars simultaneously perform independent random walks. When a car encounters an available parking spot it parks there. Other cars can still drive over the site, but cannot park there. For a large class of transitive and unimodular graphs, we show that the root is almost surely visited infinitely many times when $p \geq 1/2$, and only finitely many times otherwise. 
		 \end{abstract}

	\maketitle

	\section{Introduction}
\label{sec:intro}

The study of parking functions dates back over 50 years to the work of Konheim and Weiss \cite{KW66}. Motivated by a hashing algorithm, they introduced a parking process on the path with $n$ vertices. It starts with a parking spot at each vertex and  $x_i$ cars at spot $i$.  The configuration vector $x=(x_1,\hdots,x_n)$ has the multinomial distribution that comes from placing each of the $\lceil \alpha n \rceil$ cars independently 
at a uniformly chosen random spot. Cars then deterministically drive with unit velocity towards $1$  and park at the first available spot. Parking spots fit at most one car and ties are broken arbitrarily. 

Configurations $x$ for which all cars park are called \emph{parking functions}. The authors of \cite{KW66} find that the probability that a random initial configuration is a parking function converges to $(1-\alpha) e^{\alpha}$ as $n \to \infty$. 
Parking functions have attracted a lot of attention since. Their relation to polytopes was studied by  Stanley and  Pitman in \cite{stanley}. More recently,  Diaconis and  Hicks studied the geometry of a random parking function \cite{persi}.

In this paper, we modify the underlying graph and car trajectories so that the parking process is less combinatorial and more like an interacting particle system from statistical physics.  We consider a process in which vertices of a graph $(\mathcal{V},\mathcal{E})$ are each initially independently labeled a \emph{car} or a parking \emph{spot} with respective probability $p$ and $1-p$. The cars perform independent random walks according to a transition kernel $K$ and stop at a spot if it is unoccupied. If multiple cars arrive at the same unoccupied spot at the same time, one of them is uniformly chosen to park. Cars arriving at occupied spots pass through unaffected. For each site $v\in \mathcal{V}$, let $V_t^{(v)}$ be the number of visits to $v$ up to time $t$ by cars (not including time zero), and let $V^{(v)} = \lim_{t\rightarrow \infty} V_{t}^{(v)}$ be the total number of visits to $v$. When the $V^{(v)}_t$ have the same distribution for all $v$, we write $V_t$ and $V$ for random variables with the same distributions as $V^{(v)}_t$ and $V^{(v)}$.

At $p=1/2$, the densities of cars and spots are equal. It is natural to guess that a phase transition in $V^{(v)}$ occurs at this balance point. We show that this is  true for sufficiently homogeneous graphs, and describe the behavior at criticality under rather general conditions. Before describing these conditions, we state a special case of our result on the lattice  (see Figure \ref{fig:sim_Z}).

\begin{thm}\thlabel{thm:main_lattice}
	Consider the parking process on $\mathbb{Z}^{d}$ with simple symmetric random walks. 
	\begin{description}
		\item[(i)]  If $p \ge 1/2$, then $V$ is infinite almost surely. Moreover, $\E V_{t} = (2p-1)t+o(t)$. 
		\vspace{0.1cm}
		\item[(ii)] If $p<1/2$,  then $V$ is finite almost surely. Moreover, if $p<(256 d^6 e^2)^{-1}$, then $\E V<\infty$.
	\end{description}		
\end{thm}

\begin{figure}[h]
	\includegraphics[width=0.48\textwidth]{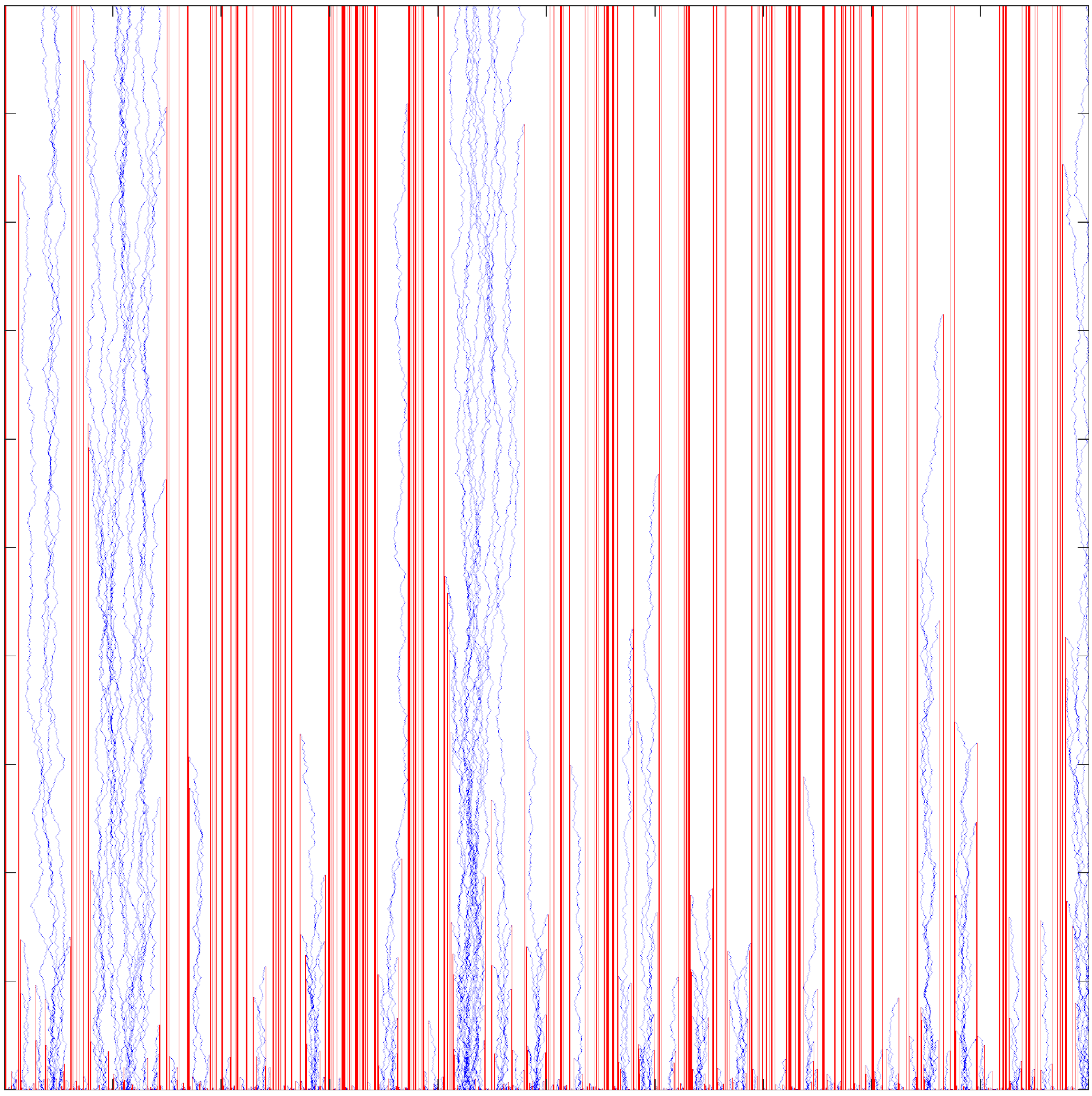}
	\hfill
	\includegraphics[width=0.48\textwidth]{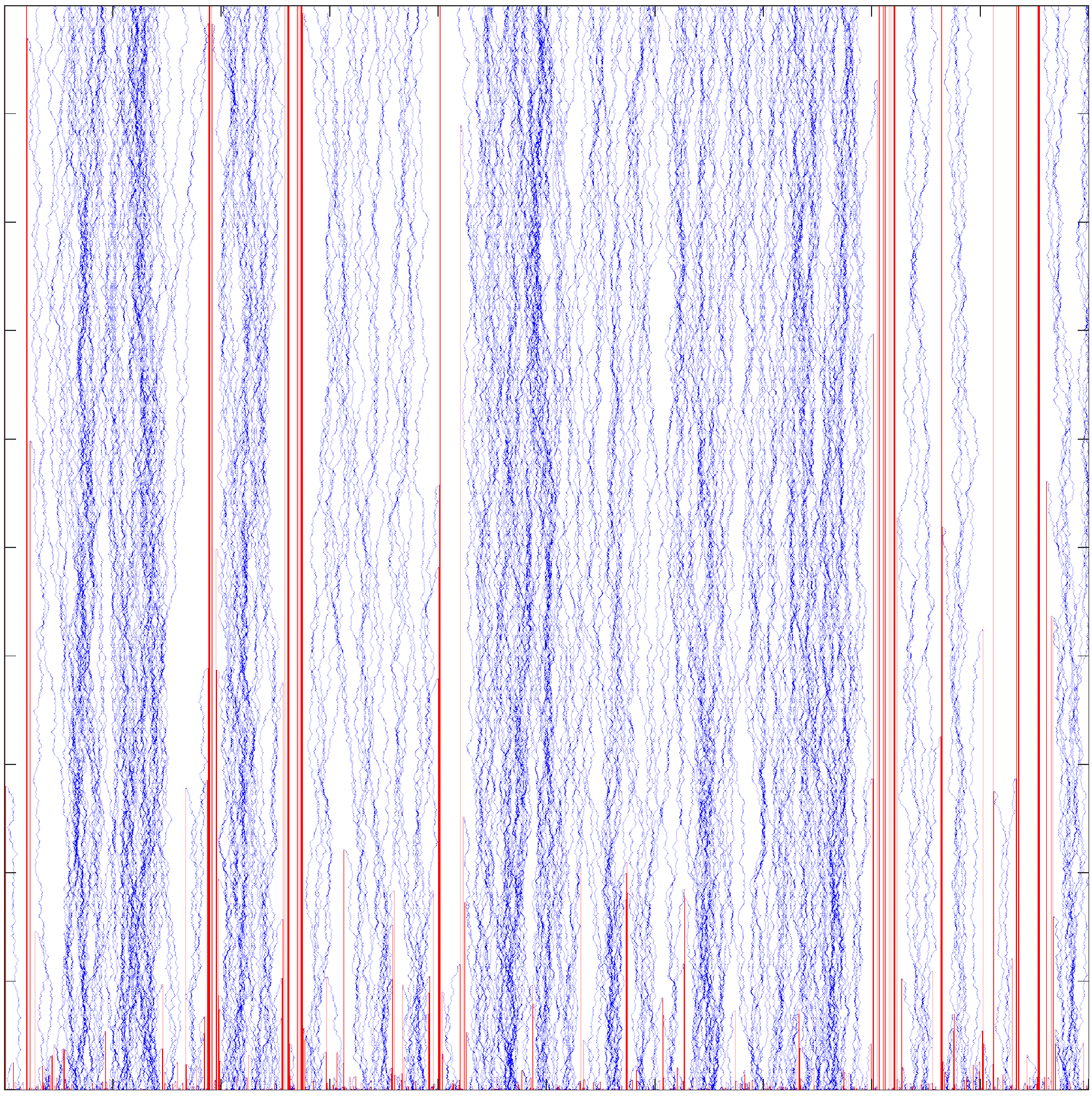}
	\caption{ Simulation of parking process on finite path of $2000$ nodes for 2000 iterations for car density $p=0.45$ (left) and $p=0.55$ (right). Cars and spots are depicted by blue and red dots, respectively. Times runs from bottom to top.  
	}
	\label{fig:sim_Z}
\end{figure}

\subsection{Background and motivation} The setting for our parking dynamics is the class of transitive, unimodular graphs with homogeneous random walk kernels. Applying unimodularity to study interacting particle systems has been a recent trend \cite{tom,muller2015interacting, CRW2}.  This level of generality often allows for simpler proofs and deeper insights about why the theorems hold.  

Cars and spots are a helpful metaphor when visualizing our process. However, the main feature of our dynamics are 
``collisions" between cars and parking spots that result in mutual annihilation. Numerous similar processes have evaded a rigorous analysis, so it may be somewhat surprising that we are able to, in some generality, characterize the phase transition in this annihilating system. We now put our work in the context of a few related models.

A model that is fairly well understood is \emph{annihilating random walk}, which can be described as the parking process in continuous time, with no parking spots ($p=1$), and in which cars mutually annihilate upon colliding. The process was introduced by Erd\"os and Ney \cite{erdos1974} in one dimension, and first analyzed on $\mathbb Z^d$ by Griffeath \cite{griffeath}, Bramson and Griffeath \cite{first}, and Arratia \cite{arratia2, arratia}. To study $p_t$, the probability that a particle is at the origin at time $t$, the authors rely on a parity coupling to \emph{coalescing random walk}, the modified process in which only one particle is annihilated in each collision. The advantage of coalescing random walk is that its time-reversal dual is 
the Markov process known as the \emph{voter model}. This
makes it possible to analyze the coalescing random walk, and thus the annihilating variant, by studying the size of a voter model cluster. 
%This is tractable, because it has the same law as a  time change %of simple, nearest neighbor random walk \cite{CRW1}. 
There does not seem to exist a Markov process with a useful
time-reversal connection to the parking process. However, one of our main tools is another kind of duality, between spots and cars, which is a consequence of a mass transport 
principle on unimodular graphs~(see Lemma \ref{lemma:mass_transport}). 
%
%%%% Rmk: Replaced the following paragraph by the next one, which seems more relevant.
%An alternative approach to studying coalescing random walk for $d \geq 6$ was developed by Kesten and Van den Berg in \cite{kesten}. They made rigorous the heuristic differential equation $p_t' = -c p_t^2$. This comes down to computing the probability that two distant particles reach $0$ at the same time, but do not meet before. This is why they require $d\geq 6$. A similar approach for the parking process seems intractable. For two distant cars to reach 0 at the same time requires the coordination of the entire system, not just two particles.
%
The model closer to ours is the \emph{two-type annihilating
	random walk}
studied by Bramson and Lebowitz~\cite{bramson1991asymptotic,BL1991,BL2001}, 
in which particles of the same type do not interact, but 
two particles of different types mutually annihilate.

The introduction of stationary particles (i.e., spots), makes
the question of survival of moving particles (i.e., cars) in our 
annihilating process resemble the notoriously difficult question of fluctuation in \emph{ballistic annihilation}. In this process
on $\mathbb{Z}$, each site is initially a particle (resp., a blockade)
independently with probability $p$ (resp., $1-p$). Each particle is also independently assigned a direction, 
either left or right, with equal probability.   Particles then move deterministically 
at unit velocity in their assigned directions and annihilate if they meet another particle or blockade. 
If every blockade is destroyed almost surely we say the process \emph{fluctuates}, otherwise it \emph{fixates}.
It is conjectured in \cite{ballistic} that for $p> 3/4$ this process fluctuates. This is worked out in a nearly rigorous way in \cite{b9}, but  the argument does not yield any probabilistic intuition. Recent papers \cite{arrows,bullet} 
prove results that imply fixation for $p< 2/3 + \epsilon$ for a small $\epsilon>0$, but currently nothing is known about the fluctuation phase, not even whether it exists for some $p<1$.

A related process is the \emph{meteor model}, introduced by Benjamini about ten years ago (see \cite{CRW}). Fix an 
$\epsilon >0$ and place $\epsilon$-balls in Euclidean space with centers at a unit intensity Poisson process. Each of these 
balls is assigned a uniformly random direction, and proceeds to
move deterministically along this direction at unit speed. Like in ballistic annihilation, when meteors collide they mutually annihilate. It remains an open problem, in any dimension $d \geq 2$, to prove that the origin is a.s. occupied by infinitely many meteors \cite[Conjecture 3.4]{CRW}.

The difficulty with models like parking, ballistic annihilation, and the meteor model is that the long-time behavior of different particles is highly correlated: knowing a particle has yet to be annihilated depends in a meaningful way on all particles in a growing neighborhood. For example, if a meteor has survived up to time $t$, then many meteors in front of it must have been destroyed earlier. This requires intricate coordination between a compounding number of meteors. Similarly, if a car $c$ is unparked at time $t$ in the parking model, then every parking spot on the trajectory of $c$  must have been occupied before $c$ arrived. It is not immediately clear how this affects the likelihood of having
spots or cars near $c$. One  might  guess  that $c$ needed many  cars  nearby  because  it  has  survived  so  long,  but,  as  time  goes  on,  these cars  protect $c$ by  parking,  and  so 
$c$ may  become  more  and  more  isolated.
The  complicated  dependence  structure,  especially  at  criticality ($p=1/2$),  is evident  in  the  clustering   of  cars  and  unparked  spots we  see  in  simulations (see  Figure~\ref{fig_evol}).

\begin{figure}[htb!]
	\centering
	\includegraphics[trim=0cm 0cm 0cm 0cm,clip, width=0.32\textwidth]{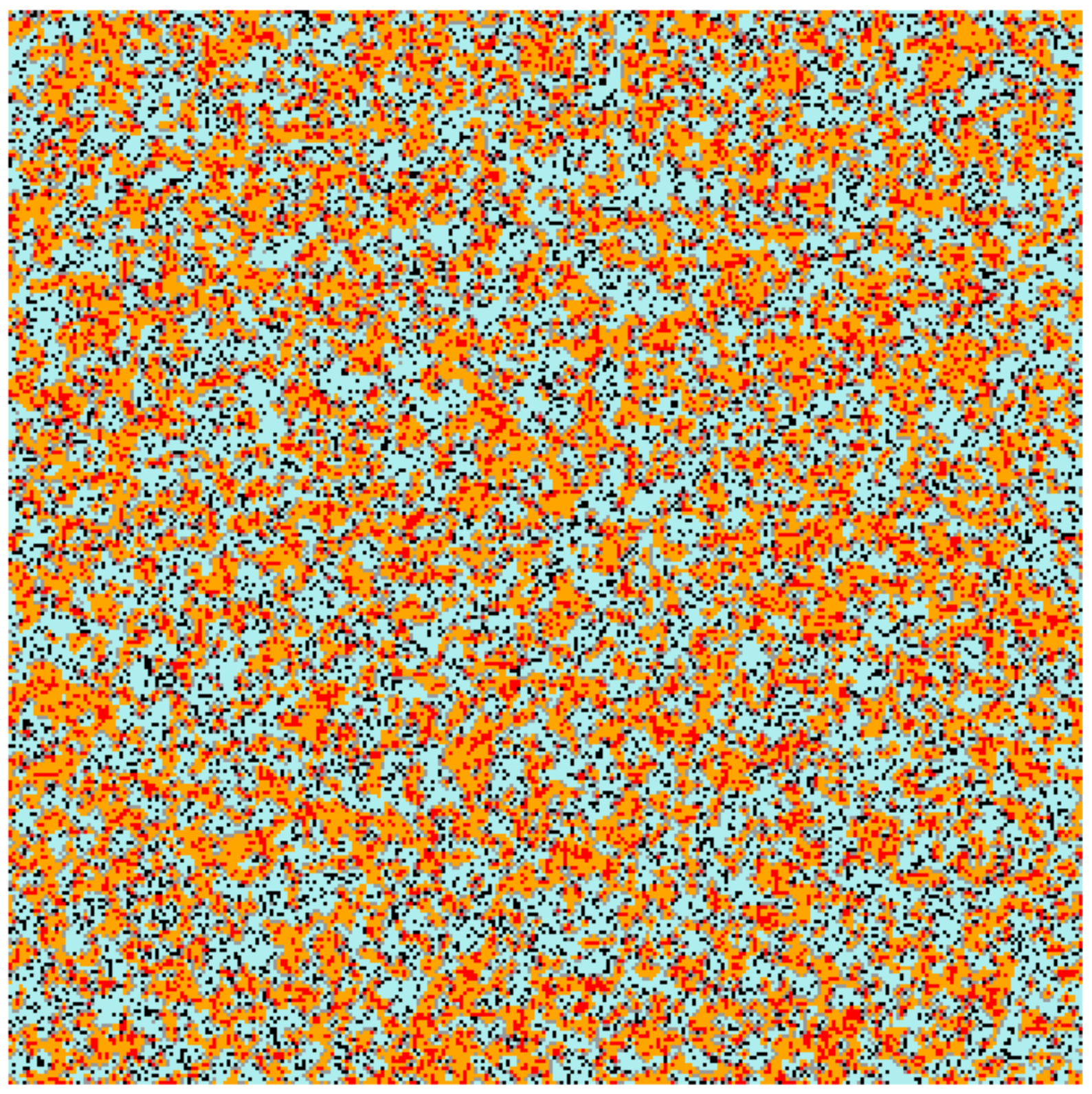}
	\includegraphics[trim=0cm 0cm 0cm 0cm,clip, width=0.32\textwidth]{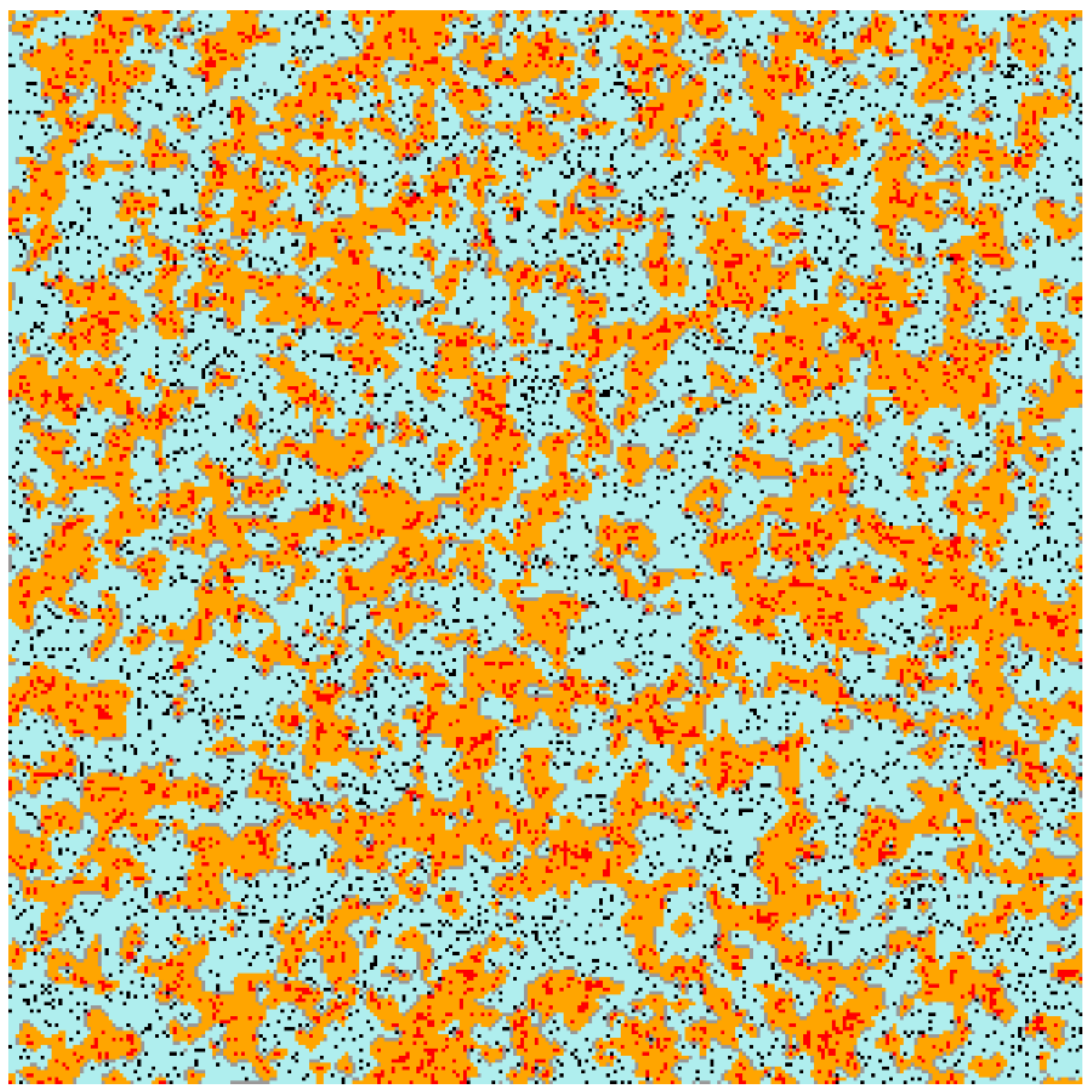}
	\includegraphics[trim=0cm 0cm 0cm 0cm,clip, width=0.32\textwidth]{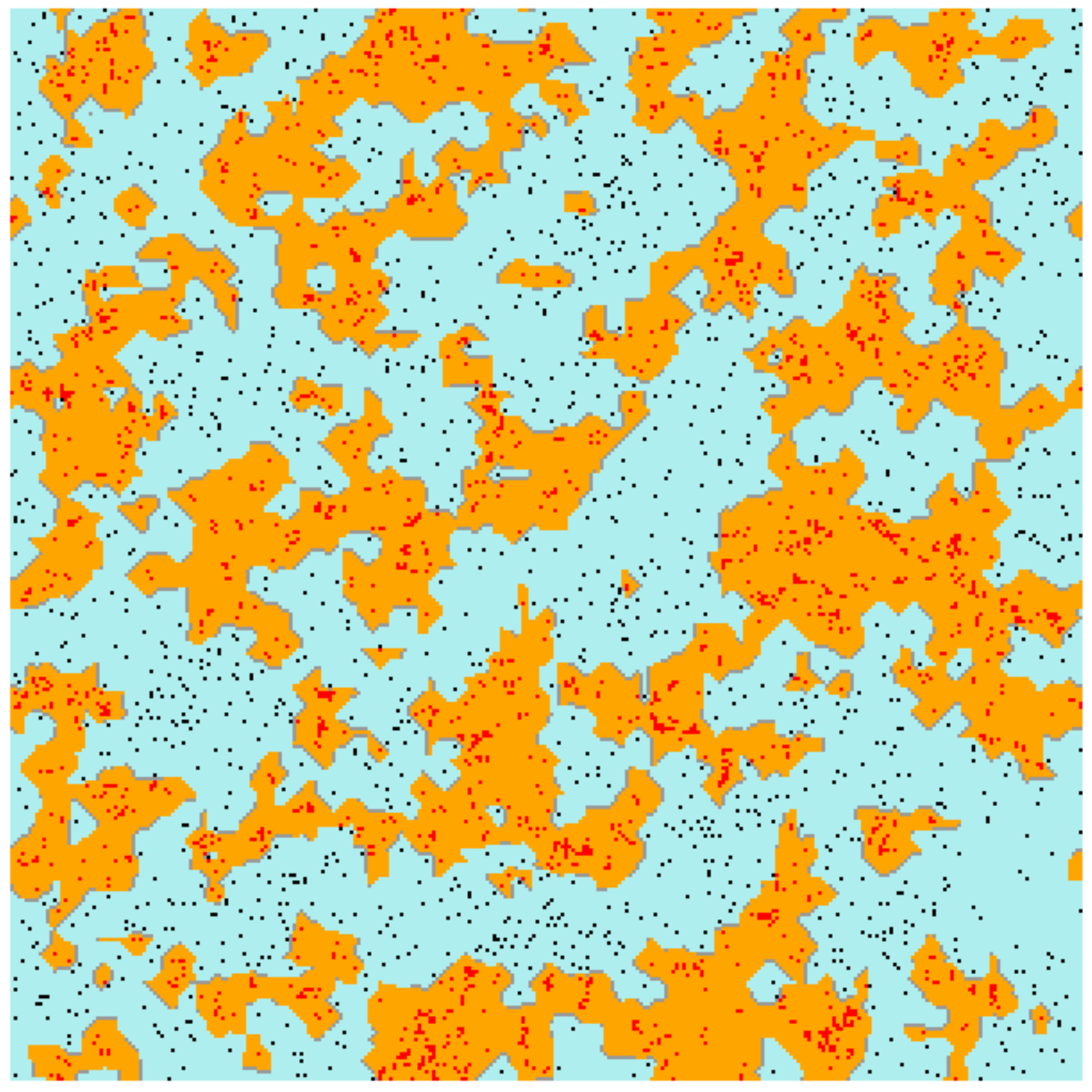}
	\caption{\label{fig_evol} 
		Evolution of the critical ($p=1/2$) dynamics on the unoriented two-dimensional lattice, 
		represented by the $300\times300$ square with periodic boundary. 
		Times $10$, $50$ and $300$ are depicted, 
		and any empty site whose closest non-empty site is a car (resp.~a parking spot) is colored blue (resp.~orange), with increasingly rare ties colored grey.}
\end{figure}

\subsection{Related work} 
%In this subsection, we discuss some of the related works in the literature.  

Goldschmidt and Przykucki study a similar process on trees in \cite{tree}. Let $\Poi(\alpha)$ denote a Poisson distribution with mean $\alpha$. Their main result concerns rooted Galton-Watson trees with a $\Poi(1)$ offspring distribution, oriented so that all edges point towards the root, and therefore all random walks move at each step closer to the root. Every vertex is initially a parking spot. Place $\Poi(\alpha)$ cars at each site. A car parks if it arrives at an available spot (breaking ties uniformly at random). If a car arrives at an occupied spot, it continues towards the root until it finds a spot or drives off the tree through the root. %denoted by $\root$. 

Let $X$ be the total number of cars that never find a spot. Since the tree (and therefore $X$) is almost surely finite, it is more natural to consider the expected value of $X$. In \cite[Theorem 1.2]{tree}, the authors prove that $\E X$ undergoes a sharp phase transition: $\E X = \infty$ for $\alpha > 1/2$, and $\E X \leq 1$ for $\alpha \leq 1/2.$ They also study $X$ when the critical Galton--Watson tree is conditioned to be infinite. In \cite[Theorem 1.3]{tree} they prove that,  for $\alpha < 1/2$, with probability $q_\alpha = \sqrt{1-2\alpha}(1-\alpha)$ every car finds a parking spot. Note that $\alpha = 1/2$ is not the exact analogue of $p=1/2$ in the parking process we consider here. Because \cite{tree} has every site initially a parking spot, the probability that a parking spot is empty after placing cars is $e^{-\alpha}$, and the expected number of cars at a site is $\E(\Poi(\alpha) -1)^+ = e^{-\alpha}+\alpha-1$. These two quantities are not equal at $\alpha =1/2$; in fact, there are many more vacant spots than cars. We will see that this imbalance comes from the fact that the oriented trees they consider are not unimodular. Goldschmidt and Przykucki's work gives a probabilistic intuition for an analogous phase transition established by  Lackner and Panholzer~\cite{uniform_parking}. There is also another class of parking models (also called random sequential adsorption), in which cars arrive by external deposition and do not move within $G$. See, for example,~\cite{DFK2008}.

%In \cite{uniform_parking}, Lackner and Panholzer consider a parking process on directed uniformly random labeled trees on $n$ vertices that closely resembles the work of Konheim and Weiss. Every vertex is initially a parking spot. Place $m\lceil \alpha n\rceil$ cars uniformly at random, and have them attempt to park as before. They show that when $\alpha \geq 1/2$ the probability every car finds a parking spot goes to 0 (as $n \to \infty$). For $\alpha< 1/2$ this probability converges to $q_\alpha$. Uniformly random trees are well-approximated by critical Galton-Watson trees. In fact, the result for uniformly random trees is a corollary of \cite[Theorem 1.3]{tree}.

For \cite{tree} and \cite{uniform_parking}, the  randomness comes from the environment and the initial placement of cars. Since the underlying graph is a directed tree, the cars follow deterministic paths. By contrast, in our parking process on graphs cars decide where to move according to a common transition kernel $K$. Depending on this kernel, the resulting car trajectories   may be deterministic or random. Our main result, Theorem \ref{thm:main},
states that \thref{thm:main_lattice} holds in general if the kernel $K$ has certain transitivity and unimodularity properties.

\subsection*{Organization of paper}
We give a more precise definition of the model and our results in Section \ref{section:definitions}. We give some examples and state open questions in Section \ref{section:simulations and further questions}. The proofs of our main results---\thref{thm:main} and \thref{thm:prob}---are given in Section \ref{sec:main}. 

\subsection*{Note added in the revision}
As we prepared the revision of this paper, we became aware of the recent paper \cite{CRS2018}, which contains a recurrence theorem for a continuous time version of the model (similar to the first half of Theorem~\ref{thm:main}(i)). Their model allows for two types of moving particles, and when type-B particles are stationary \cite[Section 4.4]{CRS2018}, it resembles our model.

\vspace{0.3cm}
\section{Definitions and statements of results}
\label{section:definitions}

Let $G=(\mathcal{V},\mathcal{E})$ be a graph, where $\mathcal{V}$ is a set of vertices and $\mathcal{E}\subseteq \mathcal{V}^{2}$ is a set of edges. We assume that $G$ is a locally finite, connected, simple graph, and we distinguish an arbitrary vertex $\root\in \mathcal{V}$ for purposes of stating some of our results. For $v\in \mathcal{V}$, let $N(v)$ be the set of all neighbors of $v$ in $G$. For $x,y\in \mathcal{V}$, we let $\mathrm{dist}(x,y)$ denote shortest path distance between $x$ and $y$. Let $\B(x,t) = \{v\in G \colon \mathrm{dist}(x,v)\le t\}$ denote the ball of radius $t$ centered at $x$, and let $\ind{\cdot}$ be an indicator random variable.

\sloppy The parking process is defined on a pair $(G,K)$, where $K$ is a \textit{(Markov) kernel} on $G$, which is a function $K\colon\mathcal{V}\times \mathcal{V}\rightarrow [0,1]$ such that $K(u,v)=0$ if $(u,v)\notin \mathcal{E}$ and $\sum_{v\in N(u)} K(u,v)=1$. For each $u,v\in \mathcal{V}$, we say $u$ is \textit{accessible} from $v$ if there exists a sequence $v=x_{0}, x_{1},\hdots, x_{n}=u$ of adjacent nodes such that $\prod_{i=0}^{n-1} K(x_{i},x_{i+1})>0$.

Initially, each $v\in \mathcal{V}$ is assigned an unparked car, written as $(v,\texttt{unparked})$, independently with probability $p$. Unassigned sites are parking spots that can fit one car. Unparked cars simultaneously perform independent discrete time random walks according to the kernel $K$ until they find an available spot. When multiple unparked cars move to the same parking spot, each car generates a uniform $[0,1]$ variable. The car with smallest value parks and the others remain unparked.

Our probability space is 
\begin{equation}\label{eq: Omega_def}
\Omega:=  \big( \{-1,1\} \times(\mathcal{V}^\mathbb{N}) \times ([0,1]^\mathbb{N}) \big)^{\mathcal{V}}
\end{equation}
with probability measure $\P_p$ under which all coordinates are independent, and for each $v\in \mathcal{V}$ the three independent components are distributed as follows. The first coordinate is a random variable with probability $p$ to be 1 (if there is a car initially at $v$) and probability $1-p$ to be $-1$ (otherwise). The second is a random walk started at $v$ with transition kernel $K$, which is the path that an unparked car placed at $v$ will follow (this path continues past the parking time). The third is a sequence of i.i.d.\ uniform $[0,1]$ random variables to break ties if multiple cars arrive at the same parking spot.

If $j \in \mathcal{V}$ initially has a car, we say that it \emph{visits} site $v$ at time $t\ge 1$ if $(j,\texttt{unparked})$ is at some neighbor $u\in N(v)$ at time $t-1$ and moves to $v$ at time $t$. Define
\begin{equation*}
V^{(v)}_{t} = \sum_{s=1}^{t} \# \{ j\in \mathcal{V}\colon \text{car $j$ visits site $v$ at time $s$}  \},
\end{equation*}   
and let $V^{(v)} := \lim_{t\rightarrow \infty} V^{(v)}_{t}$ be the total number of visits to $v$.  When the $V^{(v)}_t$ have the same distribution for all $v$, we write $V_t$ and $V$ for random variables with the same distributions as $V^{(v)}_t$ and $V^{(v)}$.
%For example, this is the case on transitive graphs.  

The canonical example of our theorem is the parking process with simple random walk on $\mathbb Z^d$ from \thref{thm:main_lattice}. In this setting we have translation invariance, a mass-transport principle by which we are able to change perspectives between cars and spots, and nice ergodicity properties. These essential features are shared by a broader class of graphs that includes regular trees and Cayley graphs. We describe them in more detail now.

Denote by $\Aut(G)$ the group of all automorphisms of the graph $G$. Let $\Aut_{K}(G)$ be the subgroup of $\Aut(G)$ consisting of all $K$-preserving automorphisms; that is 
\begin{equation*}
\Aut_{K}(G) = \{ \varphi\in \Aut(G) \colon K(u,v)=K(\varphi (u), \varphi (v) ) \quad \forall u,v\in \mathcal{V}  \}. 
\end{equation*}
Given a subgroup $\Gamma_{K}\le \Aut_{K}(G)$ of $K$-preserving automorphisms of $G$, for each $u,v\in \mathcal{V}$, denote $\Gamma_{K}(u,v)=\{ \varphi\in \Gamma_{K}\colon \varphi(u)=v \}$. We define the following conditions on the triple $(G,K,\Gamma_{K})$:
\begin{description}
	\item{1. (transitivity)}  $(G,K,\Gamma_{K})$ is \textit{transitive} if  $\Gamma_{K}(u,v)$ is nonempty for each $u,v\in \mathcal{V}$.
	\item{2. (unimodularity)} $(G,K,\Gamma_{K})$ is \textit{unimodular} if for each $u,v\in \mathcal{V}$, 
	\begin{equation*}
	|\Gamma_{K}(u,u)v|=|\Gamma_{K}(v,v)u|<\infty.
	\end{equation*}
	\item{3. (infinitely accessibility)} $(G,K,\Gamma_{K})$ is \textit{infinitely accessible} if there exist $\varphi \in \Gamma_K$ and $u\in \mathcal{V}$ such that $\{\varphi^{n}(u)\,:\, n\ge 0\}$ is infinite and $u$ is accessible from $\varphi(u)$.
\end{description} 

Any subgroup $\Gamma_{K}\le \Aut_{K}(G)$ of $K$-preserving automorphisms on $G$ acts on our probability space naturally by shifting all vertex labeled variables; that is, for each $\varphi\in \Gamma_K$ and $\omega\in \Omega$, $\varphi(\omega)(v) := \omega(\varphi^{-1}(v))$. Note that if $(G,K, \Gamma_{K})$ is transitive, then the parking process on $(G,K)$ is invariant under this action in law. In particular, for each fixed $t\ge 0$, the law of $V^{(v)}_{t}$ does not depend on $v\in \mathcal{V}$. This is crucial to obtain a recursive relation for $\E V_{t}$ (see Proposition~\ref{prop:rde}).

Duality between parking spots and cars on $\mathbb{Z}^{d}$ is obtained in a form of the well-known mass-transport principle, which holds in general for unimodular graphs. Our definition of unimodularity for the triple $(G,K,\Gamma_{K})$ is adapted from the standard unimodularity defined for Bernoulli percolation (see \cite{lyons2016probability}) in order to respect the parking process. For instance, this property is trivially satisfied when $G$ is a Cayley graph and $K$ is uniform (see Example \ref{ex:burnside} for more detail).

A key ingredient for both parts of Theorem \ref{thm:main}, is a $0$-$1$ law for $V$. This relies on the existence of an ergodic transformation $\varphi\in \Gamma_{K}$. Furthermore, in order to show $\E V<\infty$ for small $p$, we need to estimate the expected exit time of a random walk from a ball of fixed radius. The infinite accessibility for $(G,K,\Gamma_{K})$ is sufficient to make these arguments work. Note that in the definition of infinite accessibility, $u$ is accessible from all $\varphi^{n}(u)$ since $\varphi \in \Gamma_{K}$. We remark that there exists a transitive and unimodular but not infinitely accessible triple $(G,K,\Gamma_{K})$ where $G$ is infinite and locally finite. (see Example \ref{ex:burnside}).

We briefly sketch the argument that any $\varphi\in \Gamma_{K}$ in the definition of infinite accessibility must be ergodic. Take an event $A$ which is $\varphi$-invariant, approximate it by a cylinder event $B$, and choose any $k$ so that $\varphi^k B$ and $B$ depend on disjoint sets of variables and thus are independent (this is possible because all vertices have infinite orbits under $\varphi$). Since $B$ approximates $A$ and $\varphi^k B$ approximates $\varphi^k A = A$, $A$ is independent of itself.

Now we state the general version of \thref{thm:main_lattice} and also give a theorem that describes the probabilities that cars and spots remain unparked for all time.

\begin{thm}\thlabel{thm:main}
	Let $G$ be a locally finite graph with a kernel $K$, and let $\Gamma_K$ be a subgroup of $\Aut_K(G)$. If $(G,K,\Gamma_{K})$ is transitive, unimodular, and infinitely accessible, then the following hold.
	\begin{description}
		\item[(i)]  If $p \ge 1/2$, then $V$ is infinite almost surely. Moreover,  $\E V_t = (2p-1)t + o(t)$ as $t \to \infty$.
		\vspace{0.1cm}
		\item[(ii)] If $p<1/2$,  then $V$ is finite almost surely. Moreover, if $p$ is sufficiently small, then $\E V<\infty$. 
		%$(G,K)$ is infinitely escapable (instead of being infinitely accessible) and 
	\end{description}		
\end{thm}

The proof of Theorem \ref{thm:main} hinges on a recursive distributional equation (Proposition~\ref{prop:rde}) that expresses $V_{t+1}$ in terms of the number of visits to its neighboring sites, and a duality between parking spots and cars coming from the mass-transport principle (Lemma \ref{lemma:mass_transport}).

We also specify the probability that a given car eventually parks, and that a parking spot remains vacant forever. 
For the sake of concision, we say that a car \emph{parks} if it eventually parks, and we say that a spot is \emph{parked in} if a car eventually parks in that spot.
\begin{thm} \thlabel{thm:prob} 
	Let $(G,K,\Gamma_K)$ 
	%and $(\xi_{t}^{p})_{t\ge 0}$ 
	be as in \thref{thm:main}. Then,
	\begin{align}
	\P[\text{car initially at $\root$ parks}\,|\, \text{$\root$ has a car initially} ] =  \begin{cases} \f{1-p}{p}, & p > 1/2 \\	
	1, & p \leq 1/2
	\end{cases} \label{eq:car}, \\
	\P[\text{spot at $\root$ is parked in}\,|\, \text{$\root$ is a parking spot} ] 
	= \begin{cases}
	1, & p \geq  1/2\\
	\f{p}{1-p}, & p < 1/2 
	\end{cases}\label{eq:spot}.
	\end{align}
	
\end{thm}

The case $p=1/2$ is especially interesting.  Note that \thref{thm:main} says $\E V_t$ grows linearly in $t$ for $p>1/2$, but sublinearly for $p=1/2$.  Moreover, \thref{thm:prob} says that when $p=1/2$, despite the fact that $V$ is infinite, every car finds a parking spot, and every parking spot is parked in. This implies that some cars drive a very long distance to find a spot.

The phase transition at $p=1/2$ does not occur on asymmetric graphs. Indeed, \cite[Proposition 3.5]{tree} shows that on the directed binary tree the critical probability is in the interval $[1/64,1/4] \approx [.02,.25]$; see also Example~\ref{example3}. Moreover, on graphs with rapid enough degree expansion, there is no phase transition. For example, consider the tree where each vertex at distance $n$ from the root has degree $2^n$. Even when $p=1$ the induced drift is so strong that only finitely many cars will visit the root.

\vspace{0.3cm}
\section{Examples and Further questions}
\label{section:simulations and further questions}

In this section, we discuss some examples of triples $(G,K,\Gamma_{K})$ for which our main theorems apply. We also state a few further questions. 
%Our prime examples are the parking processes on lattices and regular trees.

\begin{example}[Unoriented lattices and regular trees]\label{example1}
	Let $G=(\mathcal{V},\mathcal{E})$ be either the $d$-dimensional integer lattice $\mathbb{Z}^{d}$ with nearest neighbor edges and the $(d+1)$-regular tree $\Thom$. Let $K$ be the uniform kernel on $G$; that is, $K(u,v)=K(u,w)$ if $v,w\in N(u)$.  Lattices and regular trees both have a natural notion of translations, so we let $\Gamma_{K}$ be the subgroup of all translations on $G$ (recall that we don't require $\Gamma_{K}$ to be the collection of all automorphisms). Clearly $(G,K,\Gamma_{K})$ is transitive; since translations have no fixed points, unimodularity follows trivially; powers of a single translation applied to any vertex generates an infinite ray, so infinite accessibility also holds. Hence \thref{thm:main} applies to $\mathbb{Z}^{d}$ and $\Thom$ with uniform kernel. See Figure \ref{fig:sim_lattice} for simulation results for the symmetric parking process on $\mathbb{Z}^{2}$. For this case, the bound for $\E_{p} V<\infty$ in \thref{thm:main_lattice} (ii) is $p<2^{-14}e^{-2}\approx 8.2602\times 10^{-6}$. $\blacktriangle$
\end{example}

\begin{example}[Oriented lattices]\label{example2}
	%Let $\Zo$ be the oriented $d$-dimensional integer lattice where the orientation is in the negative coordinate direction. In this example we consider the parking process on $\Zo$ with symmetric random walks. 
	Let $G$ be the $d$-dimensional integer lattice $\mathbb{Z}^{d}$ with nearest neighbor edges. Let $(e_i)_{i=1}^d$ be the standard basis for $\mathbb{Z}^{d}$, and define a kernel $K$ on $\mathbb{Z}^{d}$ by $K(x,x-e_{i})=1/d$ for each $1\le i \le d$ and $K\equiv 0$ otherwise. Since usual translations on $\mathbb{Z}^{d}$ preserves this oriented kernel $K$, we let $\Gamma_{K}$ be the group of all translations of $\mathbb{Z}^d$ as before. 
	%For each $y\in \mathbb{Z}^{d}$ let $\tau_{y}$ be the translation automorphism $\tau_{y}(x)=x+y$. 
	%Then $\Gamma_{K}=\{\text{id}\}\cup \{ \tau_{-e_{i}}\colon 1\le i \le d \}$. 
	As in the previous example, it is easy to see that $(\mathbb{Z}^{d}, K, \Gamma_{K})$ is transitive, unimodular and infinitely accessible, 
	%Moreover, $\Gamma_{K}(x,x)=\{ \text{id} \}$ for each $x\in \vec{\mathbb{Z}}^{d}$ so $(\Z,K)$ is unimodular. 
	hence \thref{thm:main} applies. %See the left side of Figure \ref{fig:sim_lattice} for simulation results for the parking process on oriented $\mathbb{Z}^2$. 
	For this case, the bound for $\E_{p} V<\infty$ in \thref{thm:main} (ii) is $p <2^{-10}e^{-2}\approx 1.3216\times 10^{-4}$. $\blacktriangle$
\end{example}

\begin{example}[Directed $d$-ary trees]\label{example3}
	Let $G=(\mathcal{V},\mathcal{E})$ be $\Thom$ (see Example \ref{example1}). To define an oriented kernel $K$ on $\Thom$, identify a bi-infinite path $\pi$, choose one of the two orientations of $\pi$, and oriented every edge of $\Thom$ towards $\pi$. For each oriented edge $(u,v)$ of $\Thom$, let $K(u,v)=1$, so each car follows a deterministic path towards infinity. If $K(u,v)=1$, we call $v$ the parent of $u$ and $u$ a child of $v$. Assume that $\Gamma_{K}\le \Aut_{K}(G)$ is such that $(G,K,\Gamma_{K})$ is transitive. Then we claim that the triple must not be unimodular for $d\ge 2$. 
	
	To see this, fix a node $x\in \mathcal{V}$ and let $x_{1},\cdots, x_{d}$ be its children. Observe that if an automorphism $\varphi\in \Gamma_{K}$ fixes a node $y\in \mathcal{V}$, then it must also fix its parent. Hence $|\Gamma_{K}(x_{1},x_{1})x|=1$. On the other hand, by transitivity, there exists $\varphi_{i}\in \Gamma_{K}$ such that $\varphi_i(x_{1})=x_{i}$ for each $1\le i \le d$. Since $\varphi_i$ preserves graph distance as well as the oriented kernel $K$, it follows that each $\varphi_{i}$ fixes $x$ and permutes its children. Hence $|\Gamma_{K}(x,x)x_{1}|\ge d\ge 2$, as claimed. Thus \thref{thm:main} does not apply when $d\ge 2$. As mentioned previously, \cite[Proposition 3.5]{tree} shows that $\P(V=\infty)>0$ for $p\ge  1/4$ when $d=2$. Indeed, observe that if there is an infinite ray containing more than a $1/2$ density of initial cars, then $\P(V=\infty)>0$. In fact, a quasi-Bernoulli percolation argument then shows that this occurs when $p>0.038$. Thus in this case $p_{c}\in [0.02,0.038]$, contrary to the transitive unimodular case in \thref{thm:main}. 
	
	For $d=1$, the pair $(G,K)$ becomes the directed one-dimensional lattice where all edges are directed from right to left. So by Example \ref{example2}, if we take $\Gamma_{K}$ to be the group of one-dimensional translations, then the triple $(G,K,\Gamma_{K})$ satisfies the assumption of our theorems. Hence \thref{thm:main} applies. Further information about the process can be obtained at criticality $p=1/2$ due to the simple topology and kernel. Namely, we have 
	\begin{equation}\label{eq:asymptotic_EV_t_directed_Z}
	\E V_{t} \sim \sqrt{2t/\pi}.
	\end{equation}
	To see this, note that by counting the number of cars against parking spots from right to left, the lifespan of a car until parking %(see \eqref{def:liftspan})
	equals the first hitting time of 0 of an associated simple symmetric random walk $(S_{n})_{n\ge 0}$, and
	\begin{equation*}
	V_{t} \overset{d}{=} M_{t}:=\max_{0\le k \le t} S_{k} \quad \forall t\ge 1.
	\end{equation*}
	Therefore, (\ref{eq:asymptotic_EV_t_directed_Z}) follows from the well-known asymptotic of the expected running maximum for simple symmetric random walk.  $\blacktriangle$
\end{example}

\begin{example}[Cayley graphs and infinite accessibility]\label{ex:burnside}
	Let $H$ be a finitely generated infinite group with set of generators $S$. Recall that the Cayley graph of $H$ (with generating set $S$) is defined by the graph $G=(\mathcal{V},\mathcal{E})$ where $\mathcal{V}=H$ and $(a,b)\in \mathcal{E}$ if and only if $ga=b$ or $a=gb$ for some $g\in S$. Then $G$ is an infinite and locally finite simple graph with common degree $|S|$. The group action of $H$ on $H$ itself gives a natural subgroup of automorphisms $\Gamma\le \Aut(G)$, which consists of all left multiplications $a\mapsto ga$. If we endow $G$ with the uniform kernel $K$ and if we let $\Gamma_{K}=\Gamma$, then $(G,K,\Gamma_{K})$ is transitive and unimodular. 
	%Note that this construction also covers our previous examples of unoriented lattices and regular trees with even degree together with the uniform kernel.   
	
	Observe that $(G,K,\Gamma_{K})$ is infinitely accessible if and only if $H$ has some element of infinite order. 
	%This was the case for $\mathbb{Z}^{d}$ since we chose $\Gamma_{K}$ to be the subgroup of translations. In general,
	This leads to the famous question in group theory that Burnside asked in 1902 \cite{burnside1902unsettled}: \textit{Is every finitely generated group, whose every element has finite order, a finite group?} In 1964, Golod and Shafarevich~\cite{golod1964class} resolved this question negatively, and since then various counterexamples with additional properties were found (e.g., \cite{zel1991solution, zel1991solution, lysenok1996infinite}). Therefore, $(G,K,\Gamma_{K})$ is transitive and unimodular but not infinitely accessible whenever we choose $H$ to be any such counterexample to Burnside's problem. $\blacktriangle$
\end{example}

When $p=1/2$, $\E V_t$ grows sublinearly. As we have seen in (\ref{eq:asymptotic_EV_t_directed_Z}), we suspect a polynomial growth $\E V_{t}\sim Ct^{\alpha}$ for some constants $C>0$ and $\alpha\in (0,1)$, which may depend on the underlying graph. On the oriented and unoriented two-dimensional lattices, simulations are given in Figure \ref{fig:sim_lattice}.% suggest $\E V_t\sim C_{1}t^{1/4}$ and $\E V_t\sim C_{2}t^{1/2}$, respectively. %However, we do not have definitive statistical support for this behavior, so we leave this as a further question.

\begin{figure}[!htb]
	
	\includegraphics[width=.45\textwidth]{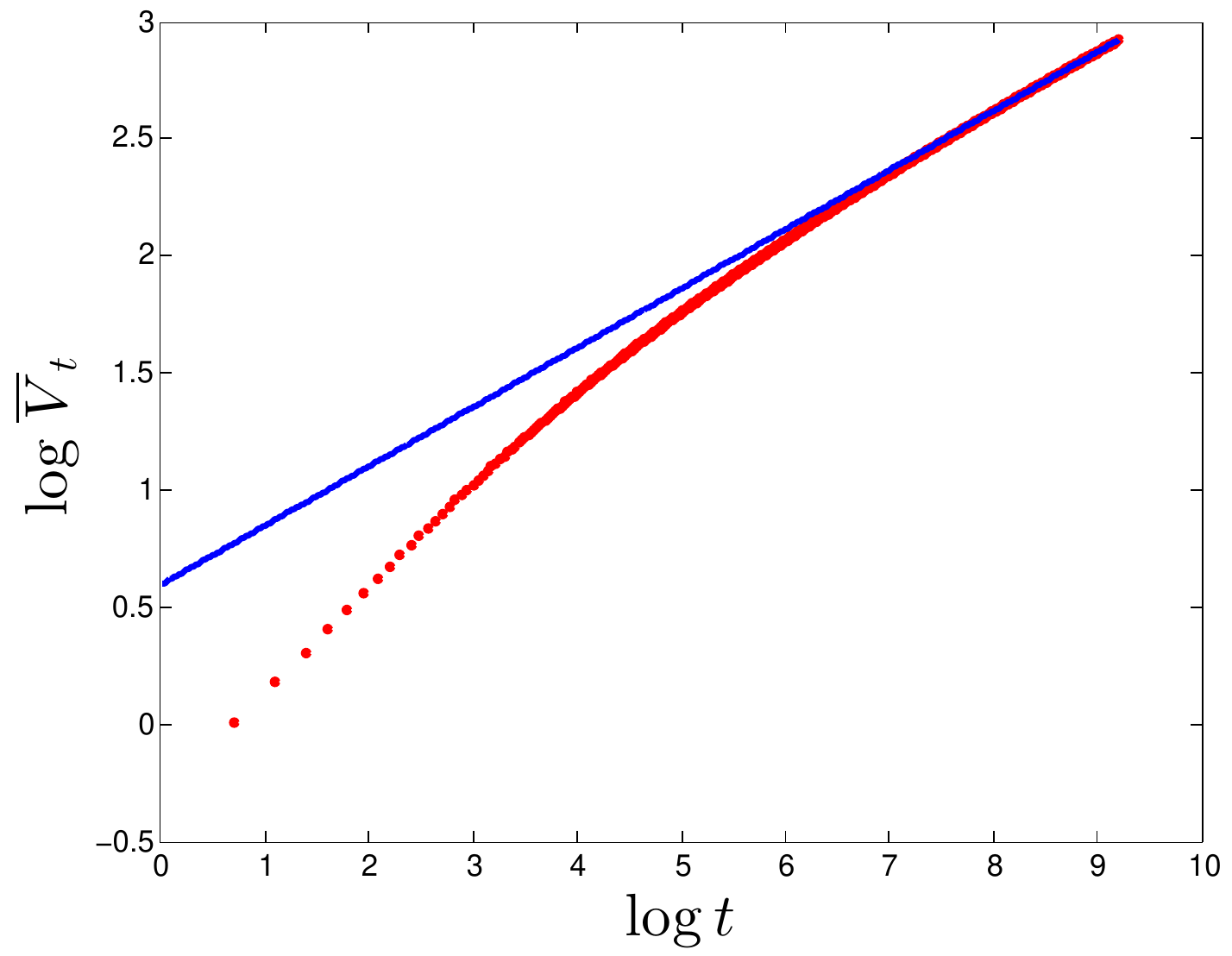}
	\hfill
	\includegraphics[width=0.45\textwidth]{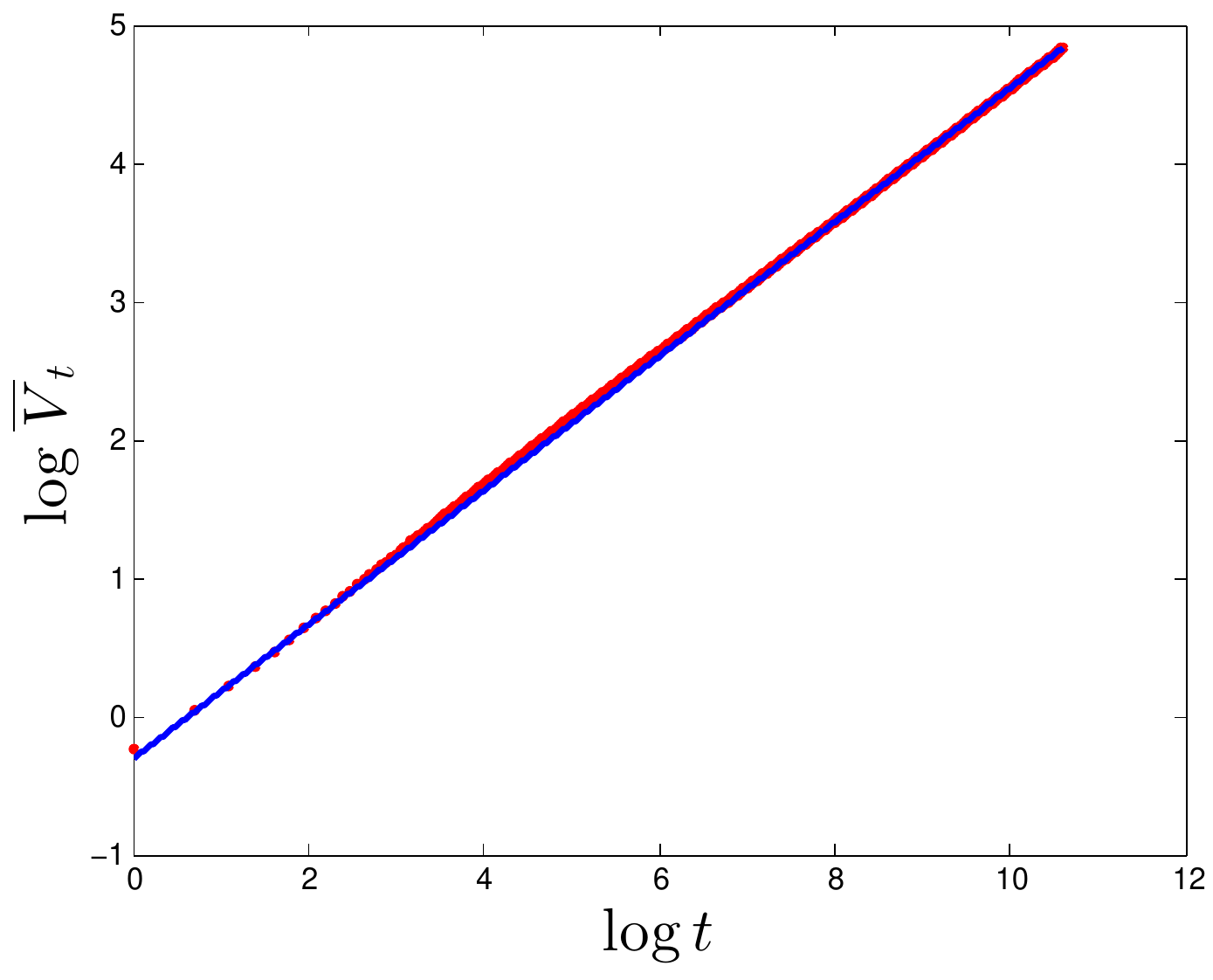}
	\caption{\label{EV_o10K} (left) Simulation of the dynamics on the 
		two-dimensional oriented lattice, which is approximated by an
		$L\times L$ square $\Lambda_L$ with periodic boundary. The 
		expectation $\E V_t$ is approximated by  $\overline V_t=L^{-2}\sum_{x\in \Lambda_L} V_t(x)$, 
		where $V_t(x)$ is the number of car-visits to $x$ in the time interval $[0,t]$. 
		The (red) plot of $\log \overline V_t$ vs.~$\log t$ up to time $t=L$ 
		is consistent with 
		a power law; the (blue) regression line is obtained by the data in the 
		time interval $[L/2,L]$ and has slope $0.2528$. This suggests that 
		$\E V_t$ might increase as $t^{1/4}$. The oriented case 
		is difficult to simulate due to slow convergence 
		and early onset of finite-size effects (which make 
		simulations past time $t=L$ questionable), so
		we do not have a definite statistical support for the $t^{1/4}$ conjecture. \\
		${}\qquad$ (right) Simulation of the dynamics on the 
		two-dimensional unoriented lattice. The notation and the value of $L$ is as in 
		Fig.~\ref{EV_o10K}, but now the 
		simulation is run up to time $t=4L$ (the larger time 
		is justified by much slower symmetric random walks),
		and the regression line is 
		obtained by the data in the 
		time interval $[2L,4L]$ and has slope $0.4854$. This suggests that
		$\E V_t$ increases as $t^{1/2}$, and that the density of blockades decreases 
		as $t^{-1/2}$, which would agree with annihilating 
		walk systems in \cite{bramson1991asymptotic}.	
	}
	\label{fig:sim_lattice}
\end{figure}

\begin{question}
	For the parking process on $(G,K)$ as before, at what rate does $\E V_t$ increase to infinity when $p=1/2$? On the oriented and unoriented $\mathbb{Z}^{2}$, is it true that $\E V_{t}\sim C_{1}t^{1/4}$ and $\E V_{t}\sim C_{2}t^{1/2}$ for some constants $C_{1},C_{2}>0$, respectively? %How do the critical exponents relate to the dimension of the lattice?
\end{question}

In \cite[Theorem 1.2]{tree} the authors observe a discontinuous phase transition for $\E X$  (where $X$ is the total number of cars that never find a spot), and find that it is bounded by 1 for $\alpha\leq 1/2$, and infinite for $\alpha>1/2$. This abrupt transition comes from the fact that a critical Galton--Watson tree is a.s.\ finite. On infinite graphs this does not occur. 
%In our case, \thref{thm:main} (i) shows an analogous result, but (ii) only gives a partial answer for $p<1/2$. Also, it is not hard to see that 
Indeed, the function $p\mapsto \E_{p} V$ is left-continuous on any $(G,K)$ because it is monotone nondecreasing (Proposition \ref{prop:monotonicity}) and lower semi-continuous (being an increasing limit of the continuous functions $p \mapsto \E_p V_t$). In particular, we have $\lim_{p\nearrow 1/2} \E_{p} V=\E_{1/2} V = \infty$, where the second equality follows from \thref{thm:main} (i). We conjecture that there is a critical exponent $\gamma>0$ such that $\E_{p} V\sim (1/2-p)^{-\gamma}$ as $p\nearrow 1/2$ (see Figure \ref{fig:EV at criticality}). Of course, one must first prove that $\E_p V$ is finite for all $p <1/2$. Our current argument for the second part of \thref{thm:main} (ii) only shows $\mathbf{E}_{p}V<\infty$ for sufficiently small $p$.
\begin{figure}[!htb]
	\minipage{0.45\textwidth}
	\includegraphics[width=\linewidth]{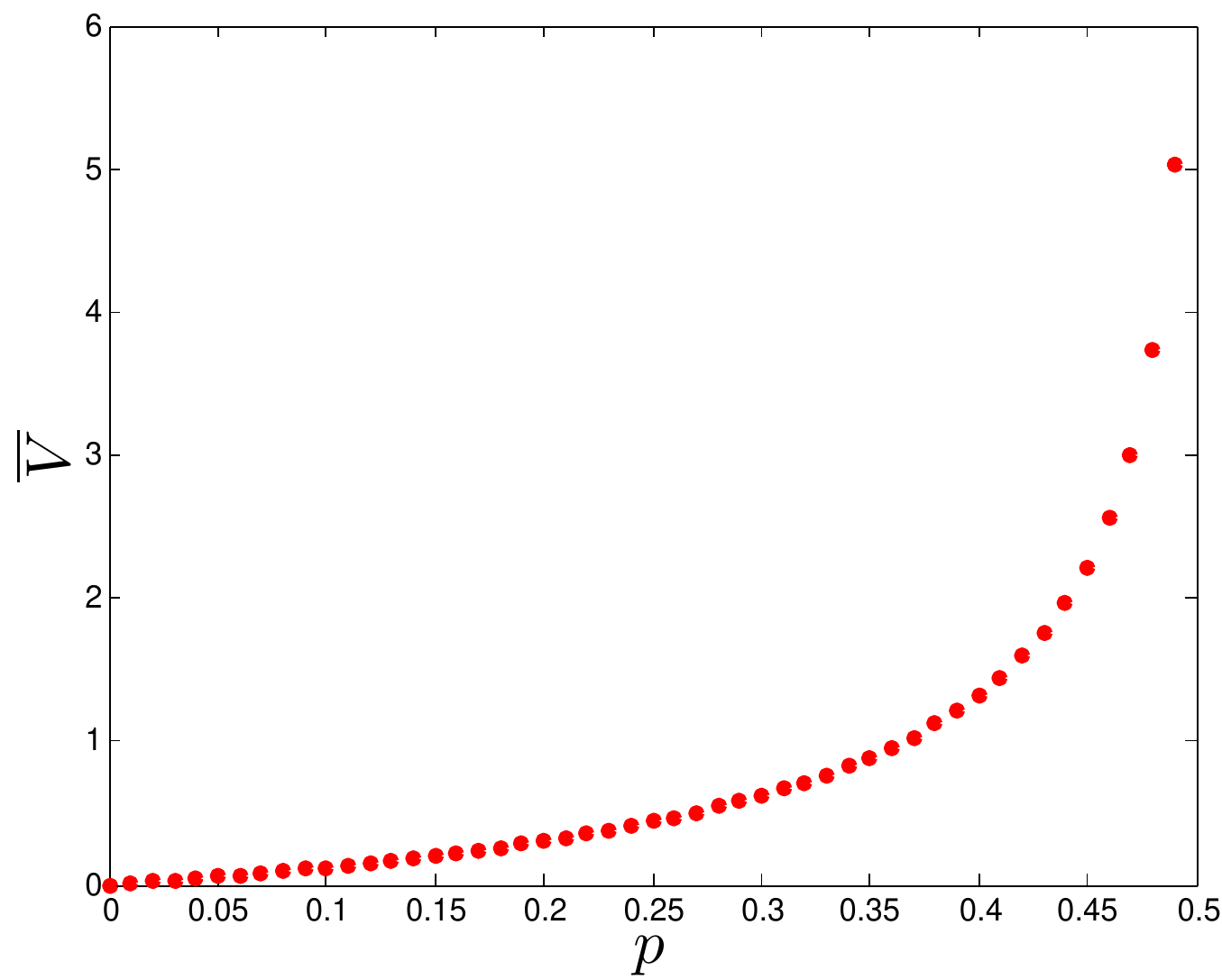}
	\endminipage\hfill
	\minipage{0.45\textwidth}
	\includegraphics[width=\linewidth]{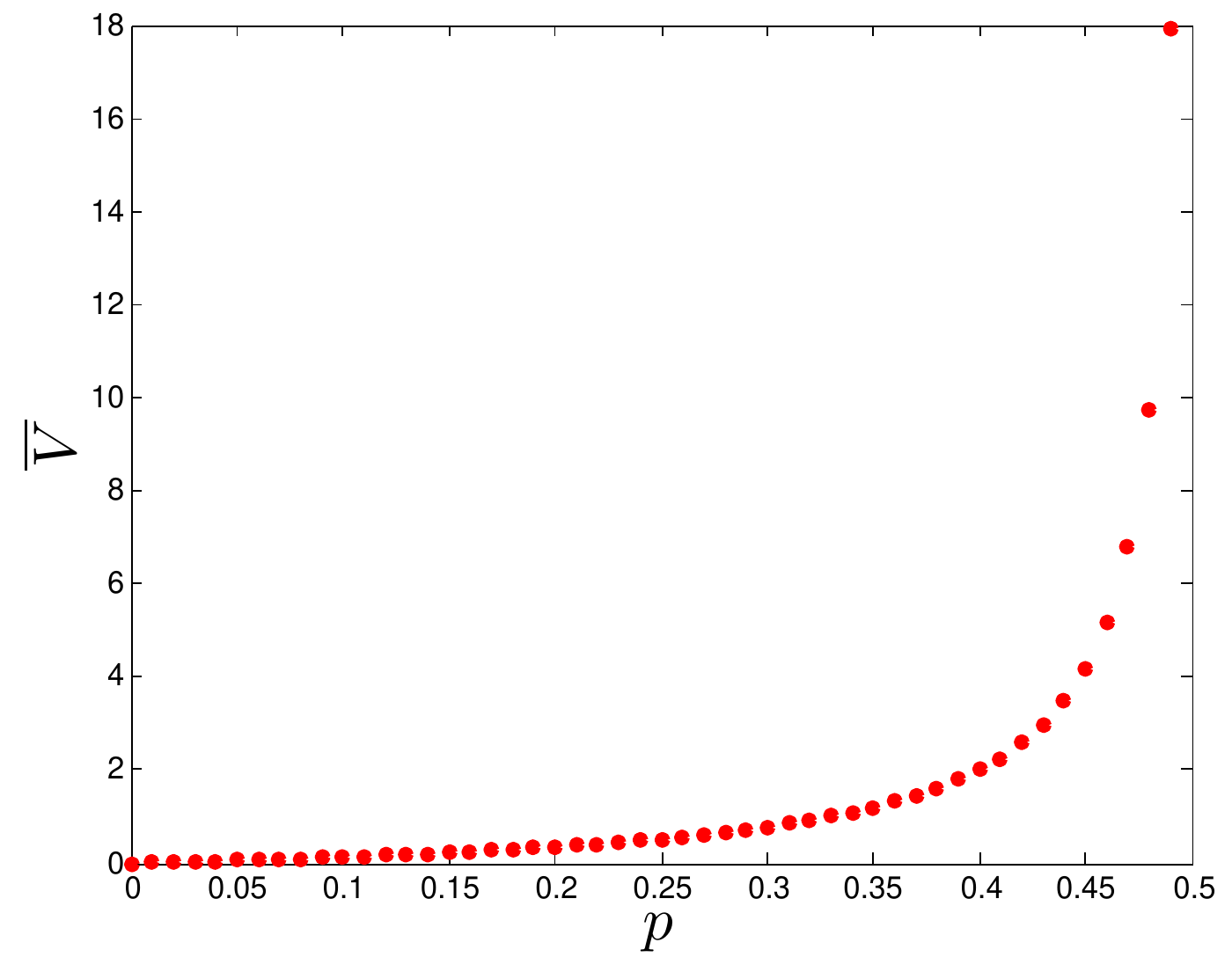}
	\endminipage
	\caption{\label{EV_osub1} Estimation of $\E V$ for subcritical density $p\in[0,0.49]$ 
		on the 
		two-dimensional oriented (left) and unoriented (right) 
		two-dimensional lattice. The dynamics are run on a square 
		$\Lambda_L$ with $L=2,000$, until time $t$ when all blockades are eliminated. 
		Then $\E V$ is approximated by $\overline V=\overline V_t$, as defined 
		in the caption of Fig.~\ref{EV_o10K}. We conjecture that 
		$\E V$ diverges as $(1/2-p)^{-\gamma}$ as $p\nearrow 1/2$, for some 
		critical exponent $\gamma$. In the unoriented case, $\gamma$ appears to be 
		near $1$, while in the oriented case, it appears to be somewhat less than $1/2$.}
	\label{fig:EV at criticality}
\end{figure}

\begin{question}
	For $(G,K,\Gamma_K)$ as in \thref{thm:main}, is $\E_{p} V <\infty$ for $p<1/2$? Also, is $\E_{p} V\sim (1/2-p)^{-\gamma}$ as $p\nearrow 1/2$, for some critical exponent $\gamma>0$?
	
\end{question}

Our final open question concerns a more precise understanding of the evolution of the parking process, and is inspired by the simulations in Figure~\ref{fig_evol}.

\begin{question}
	If $p=1/2$, is $$\lim_{t\to \infty} \P(\root \text{ is closer to a spot than to a car}) = 0?$$
	If so, how fast does this probability decay to $0$?
\end{question}

\vspace{0.3cm}
\section{Proof of \thref{thm:main} and \thref{thm:prob}}\label{sec:main}

For the entire section we assume that  $G=(\mathcal{V},\mathcal{E})$ is 
%such that $(\mathcal{V},\mathcal{E})$ is 
a locally finite infinite graph with a kernel $K$ and subgroup $\Gamma_K$ of $K$-preserving automorphisms of $G$ such that $(G,K,\Gamma_K)$ is transitive and unimodular, and that the infinite accessibility condition holds. We start by recalling the mass-transport principle, which we will use heavily in many proofs. We state our version in the following lemma, which is a minor modification of Theorem 8.7 in Lyons and Peres \cite{lyons2016probability}. 

\begin{lemma}[The mass-transport principle]\label{lemma:mass_transport}
	Let $Z\colon \mathcal{V}\times \mathcal{V}\rightarrow [0,\infty)$ be a collection of random variables such that $\E Z(x,y) = \E Z(\varphi(x),\varphi(y))$ for all $\varphi\in\Gamma_K$ whenever $y$ is accessible from $x$ or $x$ is accessible from $y$, and $\E Z(x,y) = 0$ otherwise.
	Then we have 
	\begin{equation*}
	\E \left[  \sum_{y \in  \mathcal{V} } Z(\root ,y)\right] = \E \left[  \sum_{y \in \mathcal{V}  } Z(y,\root)\right]. 
	\end{equation*}
\end{lemma}

\vspace{0.2cm}
Next, we establish a 0-1 law for $V$. 

\begin{lemma}\thlabel{lem:01}
	For all $p\in [0,1]$, we have $\P[V= \infty] \in \{0,1\}$.
\end{lemma}

\begin{proof}
	Assume that $\P(V=\infty)>0$; we will show it must then be 1. Since $(G,K)$ is infinitely accessible, there exist $\varphi\in \Gamma_{K}$ and $x\in \mathcal{V}$ such that 
	%$\gamma^{+}:=$
	$\{ \varphi^{n}(x)\,:\, n\ge 0  \}$ is infinite and $x$ is accessible from $\varphi(x)$. As noted before the statement of Theorem~\ref{thm:main}, $\varphi$ is ergodic. Therefore, since we have assumed that $\P(V=\infty)>0$, we can a.s. find a (random) integer $n_0$ such that $z := \varphi^{n_{0}}(x)$ satisfies $V^{(z)}=\infty$. Because $x$ is accessible from $z$, it suffices to show that if $x,z \in \mathcal{V}$ satisfy $K(z,x)>0$, then a.s., if $V^{(z)}=\infty$, then $V^{(x)}=\infty$. This follows as usual from the Markov property.
\end{proof}

The next proposition provides a sufficient condition for $V$ to be infinite. Namely, if $\E V_t$ grows linearly and $\E V_t^2$ grows at most quadratically, then $V$ is almost surely infinite. Under our assumptions on $(G,K,\Gamma_K)$, we will then show that this is the case for all $p > 1/2$, and use a more nuanced argument in the critical case $p=1/2$.

\vspace{0.2cm}
\begin{prop} \thlabel{prop:machine}
	If there exist $c, C>0$ such that $ \E V_{t} \geq ct$ and $\E V_{t}^{2}\le Ct^{2}$ for all $t \geq 1$, then $V$ is infinite almost surely.
\end{prop}

\begin{proof}
	The Paley-Zygmund inequality yields
	$$\P[ V_t > \E V_{t} / 2] \geq \f 14 \f{ (\E V_{t})^2 }{ \E V_t^2 } \geq \f {c^{2}}{ 4C}. $$
	Since $\E V_{t} \to \infty$ it follows that $\P[V = \infty] \ge c^{2}/4C$.  \thref{lem:01} implies that $V$ is almost surely infinite.
\end{proof}

In fact, the quadratic upper bound on $\E V_{t}^{2}$ in Proposition~\ref{prop:machine} holds for all $p\in [0,1]$. 
%We write $X \preceq Y$ to denote $X$ is stochastically dominated by $Y$ in the usual sense:  $\P[X \geq z] \leq \P[Y \geq z]$ for all $z \geq 0$.  	

\vspace{0.2cm}
\begin{prop} \thlabel{prop:second_moment}
	$\E V_{t}^{2}\le p(p+1)t^{2}$.
\end{prop}

\begin{proof}
	We dominate the parking process by a system of independent random walks with no parking. Namely, put a particle on each site if and only if there is a car initially; particles perform independent random walks indefinitely with the transition kernel $K$ used for the parking process. We couple this new process with the original parking process by letting each car follow the path of its matched particle. Let $V'_t$ be the number of visits to the origin up to time $t$ in this system (counting multiple visits by the same particle multiple times). Then by the coupling we have 
	\begin{equation}\label{eq:dominance}
	V_t \preceq V_t';
	\end{equation}
	that is, $V_t$ is stochastically dominated by $V_t'$ in the usual sense: $\P[V_t \geq z] \leq \P[V_t' \geq z]$ for all $z \geq 0$.
	
	Now for each $x,y\in \mathcal{V}$ and $t\ge 0$, let $W_{t}(x,y)$ be the number of visits of a particle at $x$ to $y$ up to time $t$. That is,
	\begin{equation*}
	W_{t}(x,y) = \sum_{s=1}^{t} \ind{ \text{a particle starts at $x$ and is at $y$ at time $s$}}.
	\end{equation*}
	Then we can write
	\begin{equation*}
	V_{t}'=\sum_{y\in \mathcal{V}} W_{t}(y,\root). 
	\end{equation*}
	A key observation is 
	\begin{equation*}
	\sum_{y\in \mathcal{V}}W_{t}(\root,y)=t \ind{\text{$\root$ has a particle initially}}.
	\end{equation*}
	Hence Lemma \ref{lemma:mass_transport} yields 
	\begin{equation*}
	\E V_{t}' = \E \left[  \sum_{y\in \mathcal{V}} W_{t}(y,\root) \right]= \E \left[ \sum_{y\in \mathcal{V}} W_{t}(\root,y)\right] = pt,
	\end{equation*}
	and
	\begin{equation*}
	\sum_{y\in \mathcal{V}}\E[W_{t}^{2}(y,\root)] = \sum_{y\in \mathcal{V}}\E[W_{t}^{2}(\root,y)] \le  \E \left[ \left( \sum_{y\in \mathcal{V}} W_{t}(\root,y) \right)^{2}\right] = pt^{2}. 
	\end{equation*}
	Now using the independence between random walk trajectories of particles starting at different sites, we have 
	\begin{eqnarray*}
		\E[ (V_{t}')^{2}] &=& \Var(V_t') + (\E V_t')^2 \\
		&\leq&\sum_{y\in \mathcal{V}}\E[W_{t}^{2}(y,\root)] + (\E V_t')^2\\
		&\leq& pt^2 + (pt)^2.
		%	\sum_{y\in \mathcal{V}}\E[W_{t}^{2}(y,\root)] + \sum_{x\ne y}     \E[W_{t}(x,\root)]\,\E[W_{t}(y,\root)] \\
		%	&\le & p(t+1)^{2} + \left( \sum_{x\in \mathcal{V}}     \E[W_{t}(x,\root)] \right) \left( \sum_{y\in \mathcal{V}} \E[W_{t}(y,\root)] \right) \\
		%	&=& pt^{2} + (\E V_{t}')^{2} = p(p+1)t^{2}.
	\end{eqnarray*}
	Hence the assertion follows from (\ref{eq:dominance}).
\end{proof}

The foundation of our analysis is a recursive formula satisfied by $\E V_{t}$. In the next result, we write ``$y$ has a car'' for ``a car initially starts at $y$.''

\begin{prop} \thlabel{prop:rde}
	For all $t\ge 0$,
	\begin{equation}\label{eq:rde}
	\E V_{t+1}-\E V_{t} = 2p -1 
	+ \P[\text{$\root$ is a spot and $V_t^{(\root)}=0$}].
	\end{equation}  
\end{prop}

\begin{proof}
	First we write  	
	\begin{equation}
	\E V_{t+1}^{(\root)}=  \sum_{y \in N(\root)} \sum_{s=1}^{t+1}\,\, \E\#\{\text{cars that visit $\root$ at time $s$ through $y$}\}.
	\end{equation}	
	By conditioning on the ``information up until time $s-1$'' and partitioning the space according to whether $y$ is an available spot at time $s-1$, we see that each unparked car at $y$ at time $s-1$ visits $\root$ independently with probability $K(y,\root)$. This gives (recall that ``spot'' is short for ``parking spot'') 
	\begin{align*}
	&\sum_{y \in N(\root)} \sum_{s=1}^{t+1} \mathbf{E}\Bigl[ K(y,\root) \#\{\text{cars visiting }y \text{ at time }s-1\} \\
	& \qquad\qquad\qquad\qquad\qquad \times \ind{y \text{ has a car or occupied spot at time $s-1$}}\Bigr] \\
	& \qquad +\sum_{y \in N(\root)}\sum_{s=1}^{t+1} \mathbf{E}\Bigl[ K(y,\root) (\#\{\text{cars visiting }y \text{ at time }s-1\}-1) \\ 
	& \qquad\qquad\qquad\qquad\qquad \times \ind{y \text{ is parked in at time }s-1}\Bigr] \\
	=~&\sum_{y \in N(\root)} K(y,\root) \Bigl[ \mathbf{E}\#\{\text{cars visiting } y \text{ at times } \leq t\} \\
	& \qquad\qquad\qquad \qquad \qquad  - \P(y \text{ is parked in at a time } \leq t)\Bigr] \\
	=~&\left( \sum_{y \in N(\root)} K(y,\root)\right) \Bigl( \mathbf{E}(V_t^{(\root)}+1)\ind{\root \text{ has a car}} \\
	& \qquad\qquad\qquad\qquad\qquad+ \mathbf{E}V_t^{(\root)} \ind{\root \text{ is a spot}} - \P(V_t^{(\root)}>0, \root \text{ is a spot})\Bigr) \\
	=~&\left( \sum_{y \in N(\root)} K(y,\root)\right) \left( \mathbf{E}V_t + 2p-1 + \P(V_t^{(\root)}=0, \root \text{ is a spot})\right).
	\end{align*}
	In the second equality we have used transitivity.
	
	Note that by the transitivity of $(G,K,\Gamma_K)$ and Lemma \ref{lemma:mass_transport}, the sum of in-probabilities at $\root$ equals 1: 
	\begin{equation*}
	\sum_{y \in N(\root)} K(y,\root) = \sum_{y \in \mathcal{V}} K(y,\root) = \sum_{y \in \mathcal{V}} K(\root,y) = 1.
	\end{equation*}
	Combining this with the above identity yields the desired recursion. 
\end{proof}

For the following discussions, it is convenient to introduce a quantity which describes the lifespan of an initial car until parking. Namely, for each $v\in \mathcal{V}$, define $\tau^{(v)}$ by  
\begin{align}\label{def:liftspan}
\tau^{(v)} := \sum_{s=1}^{\infty} \ind{ \text{a car starts at $v$ and is unparked at time $s$}}.
\end{align} 
By translation invariance of the process, the law of $\tau^{(v)}$  does not depend on $v$, so we may drop the dependence on $v$.

In the parking process adding more cars can only increase the lifespan of cars and the number of visits to a fixed site. While this is intuitively obvious, a possible concern is that introducing a new car may change the manner in which we break ties at a spot. This could potentially cause different cars to park in different places, thus shortening the paths of some cars. Of course this is not the case. We explain why in the following proposition.

\begin{prop}\label{prop:monotonicity}
	Write an arbitrary element of our probability space $\Omega = \big( \{-1,1\} \times (\mathcal{V}^{\mathbb{N}}) \times ([0,1]^\mathbb{N}) \big)^{\mathcal{V}}$ as
	\[
	\omega=(\omega_{1},\omega_{2},\omega_{3})_{v\in \mathcal{V}} = (\eta(v), (X_v(n))_{n\ge 1}, (\epsilon_v(n))_{n\ge 1})_{v\in \mathcal{V}}.
	\]
	Let $\omega,\omega'\in \Omega$ be such that $\omega_{1}(v)\le \omega_{1}'(v)$, $\omega_{2}(v)=\omega'_{2}(v)$, and $\omega_{3}(v)=\omega'_{3}(v)$ for all $v\in \mathcal{V}$. Then for all $v\in \mathcal{V}$ and $t\ge 0$, we have $\tau^{(v)}(\omega)\le \tau^{(v)}(\omega')$ and $V^{(v)}_{t}(\omega)\le V^{(v)}_{t}(\omega')$.
\end{prop}

\begin{proof}	
	%	It suffices to show the assertion for the case when $\eta(\omega) = \eta(\omega')$ except at one vertex, say, $x\in \mathcal{V}$. 
	Observe that the second assertion is implied by the first. Namely, suppose a site $y\ne x$ has a car initially which visits $v\in \mathcal{V}$ at some time $t\ge 1$ in the $\omega$-trajectory. Then $\tau^{(y)}(\omega)\ge t$, and by the first assertion, this yields $\tau^{(y)}(\omega')\ge t$. Since the trajectory of the car started at $y$ is shared in both $\omega$- and $\omega'$- trajectories, this implies that the car started at $y$ still visits $v$ at time $t$ in the $\omega'$-trajectory. Hence $V^{(v)}(\omega)\le V^{(v)}(\omega')$, as desired.    
	
	Now we show $\tau^{(v)}(\omega)\le \tau^{(v)}(\omega')$ for all $v\in \mathcal{V}$. For a contradiction, let $t \geq 1$ be the smallest time for which there is a $y$ such that $t = \tau^{(y)}(\omega') < \tau^{(y)}(\omega)$. Let $z$ be the site that the car $(y,\mathtt{unparked})$ visits at time $t$. Since $z$ is a spot and $(y,\mathtt{unparked})$ does not park at time $t$ in the $\omega$-trajectory, some other car $(u,\mathtt{unparked})$ parks at $z$ at some time $s\le t$. If $s\le t-1$, then by the minimality of $t$, $\tau^{(u)}(\omega')\ge \tau^{(u)}(\omega)=s$, so $(u,\mathtt{unparked})$ visits site $z$ at time $s$ in the $\omega'$-trajectory. Since $s\le t-1$, this means that the spot at site $z$ is already occupied before time $t$ in the $\omega'$-trajectory, which is a contradiction. Hence we may assume $s=t$, which means that $z$ is an open spot at time $t$ in the $\omega$-trajectory as well. Since $(y,\mathtt{unparked})$ parks at $z$ at time $t$ in the $\omega'$-trajectory, $\epsilon_{y}(t)$ must be minimum among all cars at site $z$ at time $t$ in the $\omega'$-trajectory. But adding more cars only adds more tie-breaking variables to be compared at site $z$ at time $t$, which would imply that the car $(y,\mathtt{unparked})$ also parks at site $z$ at time $t$ in the $\omega$-trajectory, which is a contradiction. This shows $\tau^{(v)}(\omega)\le \tau^{(v)}(\omega')$ for all $v\in \mathcal{V}$, as desired. 
	%	a new car at site $x$ only adds a new tie breaking variable $\epsilon_{x}(t)$ to be compared at site $z$ at time $t$, so this would imply that the car $(y,\mathtt{unparked})$ also parks at site $z$ at time $t$ in the $\omega$-trajectory, which is a contradiction. This shows $\tau^{(v)}(\omega)\le \tau^{(v)}(\omega')$ for all $v\in \mathcal{V}$, as desired. 
\end{proof}

%%% put Janko's thing here

A consequence of monotonicity is that if at least one car parks at a fixed site with probability one, then infinitely many do so with probability one.

\begin{lemma} \thlabel{lem:PV}
	If $\P[V=0]=0$, then $V$ is almost surely infinite.
\end{lemma}

\begin{proof}
	We prove the contrapositive. By Lemma~\ref{lem:01} we may assume that $\P[V < \infty] = 1$. Let $T$ be the smallest time after which no car visits $\root$, so that $\P[T<\infty]= 1$. Accordingly, let $t_0$ be such that $\P[T < t_0] \geq 1/2$.  Let $A$ be the event that the initial configuration has no cars in $\B(\root,t_0)=\{x \colon \mathrm{dist}(\root,x)\le t_0\}$. The event $\{T < t_0\}$ implies that no car initially outside of $\B(\root,t_0)$ ever visits $\root$, so $\{T < t_0\}\cap A \subset \{V=0\}$.  Write an arbitrary element $\omega$ in $\Omega$ as in the statement of Proposition \ref{prop:monotonicity}. 
	By the same proposition, for every fixed realization of the random walk paths $(X_{v}(n))$ and tie breakers $(\epsilon_{v}(n))$ for all $v\in \mathcal{V}$, the variables $\ind{T < t_0}$ and $\ind{A}$ are nonincreasing functions of the variables $(\eta(v))$ (which determine the vertices initially occupied by cars), so 
	\begin{align*}
	\P(V=0) &\ge \P(\{T<t_{0}\}\cap A) \\
	&= \E[ \P( \{ T<t_{0}\}\cap A \,|\, \big( (X_{v}(n))_{n\ge 1}, (\epsilon_{v}(n))_{n\ge 1}  \big)_{v\in \mathcal{V}}  \,) ] \\
	&\ge \E[ \P(T<t_{0}\,|\,(X_{v}(n)) ) \P(A \,|\, \big( (X_{v}(n))_{n\ge 1}, (\epsilon_{v}(n))_{n\ge 1}  \big)_{v\in \mathcal{V}}\, ) ] \\
	&= \P(A) \E \big[ \P\big (T<t_{0}\,|\, \big( (X_{v}(n))_{n\ge 1}, (\epsilon_{v}(n))_{n\ge 1}  \big)_{v\in \mathcal{V}} \,\big)\big] \\
	&= \P(A) \P(T<t_{0}) \ge  (1/2)(1-p)^{|\mathbb{B}(0,t_0)|},
	\end{align*}
	where we have used the FKG inequality (see \cite{FKG}) and the fact that $A$ is independent of the random walk paths.
\end{proof}

The last ingredient is another monotonicity statement that relates the probability of no visits to $\root$ conditioned on different starting configurations. 

\begin{prop}\label{prop:no visit prob.}
	For all $t\ge 0$ and $p\in (0,1)$, we have 
	\begin{equation*}
	\P(V_{t}^{(\root)}=0\,|\, \text{$\root$ has a car initially})\le \P(V_{t}=0) \le \P(V_{t}^{(\root)}=0\,|\, \text{$\root$ is a spot})
	\end{equation*}	
	and 
	\begin{equation*}
	\P(V^{(\root)}=0\,|\, \text{$\root$ has a car initially})\le \P(V=0) \le \P(V^{(\root)}=0\,|\, \text{$\root$ is a spot})
	\end{equation*}	
\end{prop}

\begin{proof}
	It suffices to show the first part. By Proposition \ref{prop:monotonicity}, 
	\begin{equation*}
	\P(V_t^{(\root)}=0 \,|\,  \text{$\root$ is a spot}) \geq \P(V_t^{(\root)}=0 \,|\, \text{$\root$ has a car initially}). 
	\end{equation*}
	Hence the assertion follows from 
	\begin{align*}
	\P(V_t^{(\root)}=0) &= \P(V_t^{(\root)}=0 \,|\,  \text{$\root$ is a spot})(1-p) \\
	&+ \P(V_t^{(\root)}=0 \,|\, \text{$\root$ has a car initially})p.
	\end{align*}
\end{proof}

We now have what we need to prove \thref{thm:main} (i).  

\begin{proof}[\textbf{Proof of \thref{thm:main} }\textup{(i)}] 
	Note that \eqref{eq:rde} with $p\in [1/2,1)$ and Proposition \ref{prop:no visit prob.} give
	\begin{align}\label{eq:above}
	\E V_{t+1} \ge \E V_{t} + (1-p)\P(V_{t}=0).
	\end{align}
	If $\P(V_t = 0)\to 0$, then $\P(V= 0 ) = 0$ and $V$ is almost surely infinite by \thref{lem:PV}. If $\P(V_t = 0) \to \delta >0$, then \eqref{eq:above} implies that $\E V_{t}  \geq \delta (1-p) t$ for all $t \geq 1$. Thus, \thref{prop:machine} implies $V$ is infinite almost surely.  
	
	To show the second part, let $p\geq 1/2$. Then 
	\begin{equation} \label{no_visits}
	\begin{aligned}
	\P(V_t^{(\root)} = 0, \root \text{ is a spot}) \leq&\ \P(V_t=0) \\
	\to&\ \P(V=0) \leq 1 - \P(V=\infty) = 0.
	\end{aligned}
	\end{equation}
	The recursion in \eqref{eq:rde} gives 
	\begin{equation*}
	\E V_t = (2p-1) t + \E V_{0} +\sum_{s=0}^{t-1} \P(V_{t}^{(\root)}=0, \text{$\root$ is a spot}), 
	\end{equation*}
	and \eqref{no_visits} implies that the last summation is of order $o(t)$.  
\end{proof}

Before we prove \thref{thm:main} (ii) we use unimodularity to relate the probability a car eventually parks to the probability a parking spot is eventually parked in.

\begin{lemma} \thlabel{lem:relate} For any $p\in (0,1)$,
	\begin{align*}
	\P[\text{a car initially at $\root$ parks}] = \P[\text{a spot at $\root$ is parked in}].  \label{eq:relate}
	\end{align*}
\end{lemma}

\begin{proof}
	For any two sites $x,y\in \mathcal{V}$, let  
	\begin{equation*}
	Z(x,y) = \ind{\text{a car starts at $x$ and it parks at $y$}}.
	\end{equation*}	
	Then by Lemma \ref{lemma:mass_transport}  
	\begin{equation*}\label{eq:pf_fixation}
	\E \left[ \sum_{y \in G} Z(\root,y) \right]= \E \left[ \sum_{y\in G} Z(y,\root)\right].
	\end{equation*}
	Since at most one car parks in each spot, the left hand side equals the probability that a car starts at $\root$ and it eventually parks. On the other hand, the right hand side equals the probability that $\root$ is a parking spot and some car parks there. This proves the assertion.
\end{proof}

The formulas for the probability that a car parks and the probability that $V=0$ are quick consequences of \thref{lem:relate} and \thref{thm:main}.

\begin{proof}[\textbf{Proof of \thref{thm:prob}}]
	First let $p\in [1/2,1)$.  Then \thref{thm:main} (i) implies $\P[V= \infty] =1$, and so $\P(V^{(\root)} = 0, \root \text{ is a spot}) \leq \P(V=0) = 0$. This implies $\P(V^{(\root)}>0 \mid \root \text{ is a spot})=1$, which is the first part of \eqref{eq:spot}. Moreover, applying this to  the relation in \thref{lem:relate} gives
	$$\P[\text{car at $\root$ parks} \,|\, \text{$\root$ has a car initially}\,] = \f{1-p}{p}$$
	for $1/2 \le p \le 1$, which is the first part of \eqref{eq:car}. Note that this probability is 1 at $p=1/2$. Monotonicity of the process ensures that the probability remains $1$ for $p <1/2$. This establishes the second part of \eqref{eq:car}. It remains to show \eqref{eq:spot} for $p\in (0,1/2]$. In this case, \eqref{eq:car} and \thref{lem:relate} yields
	$$\P[V>0\,|\, \text{$\root$ is a spot}]  = \f{p}{1-p}. $$
\end{proof}

Now we turn our attention to Theorem \ref{thm:main} (ii). 
%The first part follows from Theorem 2.2. Indeed, if $p<1/2$, then $\P(V>0 \mid 0 \text{ is a spot}) < 1$, and so $\P(V = \infty) < 1$. By Lemma~\ref{lem:01}, $\P(V=\infty)=0$. Therefore we focus on the second part of Theorem 2.1 (ii), and accordingly assume that $(G,K)$ is infinitely escapable. 
An easy application of the mass-transport principle allows us to write $\E V_{t}$ in terms of survival probabilities $\P(\tau\ge s)$ of a car. 

\begin{prop}\label{prop:EV_t and survival prob.}
	For any $t\ge 1$, 
	\begin{equation*}
	\E V_{t}= \sum_{s=0}^{t} \P[\tau\ge s].
	\end{equation*}
\end{prop} 

\begin{proof}
	For each $x,y\in \mathcal{V}$, define a random variable $Z_{t}(x,y)$ by 
	\begin{equation*}
	Z_{t}(x,y) = \ind{ \text{a car is at $x$ initially and visits $y$  at time $t$}}.
	\end{equation*}
	Fix $\root\in \mathcal{V}$. Observe that 
	\begin{equation*}
	\sum_{y \in  \mathcal{V}} Z_{t}(y,\root) = \# \{ \text{cars visiting $\root$ at time $t$} \} = V_{t}^{(\root)} - V_{t-1}^{(\root)}. 
	\end{equation*}	
	On the other hand, since each car parks at at most one spot,
	\begin{equation*}
	\sum_{y \in  \mathcal{V}} Z_{t}(\root,y) = \ind{ \text{a car starts at $\root$ and is unparked at time $t$}}.
	\end{equation*}
	Hence by Lemma \ref{lemma:mass_transport}, for all $t\ge 1$,
	\begin{equation*}
	\E V_{t} - \E V_{t-1} = \P[\text{a car starts at $\root$ and is unparked at time $t$}] = \P[\tau \ge t].
	\end{equation*}
	Thus the assertion follows. 
\end{proof}

Next, we need an estimate of the expected time that a random walk spends in a ball of fixed radius. Given a transitive triple $(G,K,\Gamma_K)$ and $v\in \mathcal{V}$, let $( X^{(v)}_{t})_{t\ge 0} \subseteq \mathcal V$ be a random walk trajectory given by the kernel $K$ of a particle initially at $v\in \mathcal{V}$. For each $j\ge 1$, define 
\begin{equation}\label{eq:hitting_time}
\mathtt{t}^{(v)}(j) = \inf\{ t\ge 0\colon \text{dist}(X^{(v)}_{0},X^{(v)}_{t})> j  \},
\end{equation}
the first exit time of $\mathbb{B}(v,j)$. Since $(G,K,\Gamma_K)$ is transitive, the law of $\mathtt{t}^{(v)}(j)$ does not depend on $v$. Define the following generating function 
\begin{equation}\label{eq:generating_ft_Def}
F(s) = \sum_{j=0}^{\infty} \E[ \mathtt{t}^{(v)}(j) ]\,s^{j}. 
\end{equation}
Finally, define $K_{\min}$ to be the minimum (non-zero) transition probability over all of $G$:
\begin{equation*}
K_{\min} = \min\{K(x,y) :  x,y\in \mathcal{V}  \text{ and }K(x,y)>0\}.
\end{equation*}
Note that since $(G,K,\Gamma_K)$ is transitive, and $G$ is locally finite, we have $K_{\min} = \min\{K(\root,y) : K(\root,y)>0\}>0$.
\begin{prop}\label{prop:generating_ft}
	$F(s)<\infty$ for all $|s|<(K_{\min})^2$.
\end{prop}

\begin{proof}
	For each $k,j\ge 0$, define the hitting probability
	\begin{equation*}
	a_{k,j} = \P\bigg[ \max_{0\le i \le k} \text{dist}(X^{(v)}_{0},X^{(v)}_{i}) = j \bigg]
	\end{equation*}
	and its generating function 
	\begin{equation*}
	Q(u,s) = \sum_{k,j\ge 0} a_{k,j} u^{k}s^{j}.
	\end{equation*}
	Note that $Q$ is well-defined on  $(-1,1)^{2}$. 
	
	We next claim that there exists a sequence $(u_m)_{m\ge0}$ such $u_m$ is accessible from $u_0$ and $\text{dist}(u_0,u_m)= m$ for every $m$. To see why this holds, let $\varphi \in \Gamma_K$ be such that $\{\varphi^n(u_0) : n \geq 0\}$ is infinite and $u_0$ is accessible from $\varphi(u_0)$. Here, we have used the definition of infinite accessibility along with transitivity. Then the set $\{\varphi^{-n}(u_0) : n \geq 0\}$ is also infinite, and $\varphi^{-n}(u_0)$ is accessible from $u_0$. For a given $m$, pick $n$ such that $\text{dist}(u_0, \varphi^{-n}(u_0)) \geq m$. Then there exists a sequence $u_0 = v_0, v_1, \dots, v_k = \phi^{-n}(u_0)$ such that $K(v_i,v_{i+1}) > 0$ for each $i$. Since $|\text{dist}(v_{i+1},u_0) - \text{dist}(v_i,u_0)| \leq 1$, we can then select $v_i$ such that $\text{dist}(v_i,u_0) = m$ and we set $u_m = v_i$. 
	
	Using the Markov property of the random walk $X_{t} = X_t^{(u_0)}$ on $(G,K)$ it holds for any $\alpha \in \Gamma_K$ that
	\begin{eqnarray*}
		\P[X_{n}^{(\alpha(u_0))}=\alpha(u_n)] =\P[X_{n}=u_n] 
		\ge (K_{\min})^{n}.
	\end{eqnarray*}
	By the triangle inequality and the Markov property, this yields  
	\begin{eqnarray*}
		a_{k,j} &\le& \prod_{\ell=0}^{\lfloor k/(2j+1) \rfloor-1} \P\big[\text{dist}(X^{(v)}_{(2j+1)\ell},X^{(v)}_{(2j+1)(\ell+1)}) \le 2j \big] \\
		&\le& (1-(K_{\min})^{2j+1})^{\frac{k}{(2j+1)}-1}
		\le 3\exp\left(-\frac{(K_{\min})^{2j+1}k}{(2j+1)}\right)
	\end{eqnarray*}
	for any $k,j\ge 0$. Hence by dominated convergence 
	\begin{equation*}
	Q(1,s)= \sum_{j=0}^{\infty} \left( \sum_{k=0}^{\infty} a_{k,j}\right)s^{j} \le 3\sum_{j=0}^{\infty} \frac{s^{j}}{1-\exp(-\frac{(K_{\min})^{2j+1}}{2j+1})}.
	\end{equation*}
	The power series on the right converges whenever $|s|<(K_{\min})^{2}$. Now observe that
	\begin{eqnarray*}
		\sum_{j=0}^{\infty} \E(\mathtt{t}^{(v)}(j)) s^{j} 
		&= & \sum_{j=0}^{\infty} \left(  \sum_{k=0}^{\infty}\P(\mathtt{t}^{(v)}(j)>k) \right) s^{j} \\
		&\le &\sum_{j=0}^{\infty} \left(  \sum_{i=0}^{j} \sum_{k=0}^{\infty} a_{k,i} \right) s^{j} 
		=\sum_{i=0}^{\infty} \sum_{k=0}^{\infty} \left(  a_{k,i}\sum_{j\ge i} s^{j} \right) \\
		&=& \sum_{i=0}^{\infty}  \left( \sum_{k=0}^{\infty} a_{k,i}\right) \frac{s^{i}}{1-s} 
		=  \frac{Q(1,s)}{(1-s)}.
	\end{eqnarray*}
	This shows $F(s)<\infty$ whenever $s\in (-(K_{\min})^{2},(K_{\min})^{2})$.
\end{proof}

Before the proof of Theorem \thref{thm:main} (ii), we need one last deterministic lemma that gives a necessary condition for a car to survive up to time $t$. We call a finite set $H\subset \mathcal{V}$ \emph{busy} if $H$ is connected and there are at least as many cars as spots initially on $H$. 

\begin{lemma}\label{busy_subgraph}
	Let $t\ge 1$. For each $\omega \in \{\tau^{(v)} \ge t\}$, there is a busy set $H = H(\omega)$ such that $H \subseteq \mathbb{B}(v,2t)$ and $H$ contains the trajectory of the car started at $v$ up to time $t$.
\end{lemma}

\begin{proof}
	For this proof, we consider an extension of our space of outcomes (defined in \eqref{eq: Omega_def}) in which the initial configuration, $\omega_1$, consists of cars, parking spots, and initially parked-in spots. Namely, we use the space $\big( \{-1,0,1\} \times(\mathcal{V}^\mathbb{N}) \times ([0,1]^\mathbb{N}) \big)^{\mathcal{V}}$, where the second and third coordinates are the same as before (random walks and tie-breaking variables), but the first coordinate has an additional state 0, which indicates an initially parked-in spot. The dynamics of the process are similar to those from before: cars follow their random walk trajectories and attempt to park in empty spots, but now they never park in initially parked-in spots.
	
	Fix an outcome 
	$(\omega_1,\omega_2,\omega_3)=(\omega_1,\omega_2,\omega_3)_{v \in \mathcal{V}} \in 
	\{\tau^{(v)} \ge t\}$ with the property that $\omega_1(x) \in \{-1,1\}$ for all $x$. 
	For a set $B\subset V$, let 
	$\omega_1^B$ agree with $\omega_1$ on $B$ and
	contain only parking spots on $B^c$. 
	Let $H_0\subset V$ be an inclusion-minimal set
	of vertices that initially contain cars which causes the event to occur; 
	that is, 
	\begin{itemize}
		\item every $x\in H_0$ is initially occupied by a car $(\omega_1(x)=1)$, 
		\item $(\omega_1^{H_0},\omega_2,\omega_3)\in  \{\tau^{(v)} \ge t\}$, but 
		\item $(\omega_1^{H_0'},\omega_2,\omega_3)\notin  \{\tau^{(v)} \ge t\}$
		for all $H_0'\subsetneqq H_0$.  
	\end{itemize}
	%Replace $\omega_1$ by $\omega_1^{H_0}$.
	
	Let $H_2\subseteq H_0$ be an inclusion-maximal set of vertices that initially contain cars in $\omega_1^{H_0}$ such that $ \{\tau^{(v)} \ge t\}$ occurs
	even if the states (in $\omega_1^{H_0}$) of all vertices in $H_2$ are replaced by state 0 (initially parked-in).
	Let $H_1=H_0\setminus (H_2 \cup \{v\})$. Last, we define the configuration
	$\widehat\omega_1$, which has cars at sites in $H_1$, initially parked-in spots at sites in $H_2$, and parking spots elsewhere. We will now argue that $H$ is busy in $\widehat\omega_1$; this will complete the proof of the lemma, since $\widehat\omega_1 \leq \omega_1$. 
	
	For any $z\in H_1$, 
	the car at $z$ must park at some site $\sigma(z)$ by time 
	$t$ (by maximality of $H_2$), and its trajectory from $z$ to $\sigma(z)$ 
	cannot leave  $\mathbb{B}(v,2t)$. The last assertion follows because the configuration inside $\mathbb{B}(v,t)$ up to time $t$ is unaffected by any changes (at any time) to the configuration outside $\mathbb{B}(v,2t)$, and so the initial location of any car that exits $\mathbb{B}(v,2t)$ before time $t$ can be replaced by an initially parked-in spot. This violates maximality of $H_2$.
	
	Define $H=H_1\cup\sigma(H_1)\cup H_2 \cup \{v\}$. 
	Then, for $z\in H_1$, all sites on the
	trajectory from $z$ to $\sigma(z)$ are in $H$. Indeed, 
	every site $u$ on this trajectory must in $\widehat\omega_1$
	either contain 
	a car (so $u\in H_1$), an initially parked-in spot (so $u\in H_2$),
	or a parking spot
	(so $u\in\sigma(H_1)$). The same argument shows that 
	all sites on 
	the ($t-1$)-step trajectory of the car initially at $v$ are in $H$. 
	By minimality of $H_0$, $H$ must be a connected set that includes this trajectory. Because $H_2$ contains no vertices which are initially cars or parking spots, and for each vertex in $\sigma(H_1)$ (initially a parking spot), there is a unique corresponding vertex in $H_1$ (initially a car), $H$ is busy in $\widehat\omega_1$.
	%Finally, use the fact that $\sigma$ is one-to-one, and that sites in $H_2$ are cars in $\omega_1$ to conclude 
	%that $H$ is busy.  
\end{proof}

We are now ready to show \thref{thm:main} (ii).

\begin{proof}[\textbf{Proof of \thref{thm:main}} \textup{(ii)}]
	Suppose $0\le p<1/2$. Rewrite the relation in \thref{lem:relate} as 
	$$\P[V>0\,|\, \text{$v$ is a spot}\,] = \f{p \P[\text{car at $v$ parks}\,|\, \text{$v$ has a car}] }{1-p} \leq \f{p}{1-p}<1.$$
	So the complementary event has positive probability: $$\P[V=0\,|\, \text{$v$ is a spot}\,]>0.$$ By Proposition \ref{prop:no visit prob.}, this implies $\P[V=0]>0$. The $0$-$1$ law in \thref{lem:01} then requires that $\P[V=\infty]= 0$.

	Next, we show $\E V<\infty$ when $p$ is small. 
	Recall that $G$ has at most $(e\Delta)^{j}$ connected subgraphs of size $j$ containing $v$, where $\Delta$ is the maximum degree of the graph, and by a Chernoff bound for Binomial$(j,p)$ variable, the probability of a connected subgraph of size $j$ being busy is at most $(2\sqrt{p(1-p)})^{j}$ when $p<1/2$. Indeed, if $Z_1, \dots Z_j$ are i.i.d. random variables with $\P[Z_i=1]=p= 1-\P[Z_i=-1]$, then this probability is equal to
	\begin{align*}
	\P[Z_1 + \dots + Z_j \geq 0] = \P \left[ \prod_{i=1}^j e^{\alpha Z_i} \geq 1\right] &\leq \left( \E e^{\alpha Z_1}\right)^j \\
	&= \left( pe^{\alpha} + (1-p)e^{-\alpha}\right)^j,
	\end{align*}
	for any $\alpha \geq 0$. Putting $\alpha = \frac{1}{2} \log \frac{1-p}{p}$, we obtain the bound $(2\sqrt{p(1-p)})^j$. In addition to this inequality, we note that $|\mathbb{B}(v,r)|\le (\Delta-1)^{r+1}$. Let $(X_{t})_{t\ge 0}$ be an independent random walk trajectory on $(G,K)$ with $X_{0}=v$. Applying Lemma~\ref{busy_subgraph} and a union bound gives 
	\begin{align*}
	& \P[\tau^{(v)}\ge t] \\
	&\le \sum_{j=1}^{\Delta^{2t+1}} \sum_{\substack{H\text{ connected}\\ |H|=j, v\in H}} (2\sqrt{p(1-p)})^{j} \, \P\big[ \text{$X_{k}$ is in $H$ for all $0\le k \le t$}  \big] \\
	&\le \sum_{j=1}^{\Delta^{2t+1}}  (2e\Delta\sqrt{p(1-p)})^{j} \, \P\big[ \text{$X_{k}$ is in $\mathbb{B}(v,j)$ for all $0\le k \le t$}  \big].
	\end{align*}
	
	Let $\mathtt{t}^{(v)}(j)$ and $F(t)$ be as defined in \eqref{eq:hitting_time} and \eqref{eq:generating_ft_Def}, respectively. Then Proposition \ref{prop:EV_t and survival prob.} and the above bound on $\P[\tau^{(v)} \geq t]$ yield
	\begin{eqnarray*}
		\E V^{(v)} 
		&\le & \sum_{t=0}^{\infty} \sum_{j=1}^{\Delta^{2t+1}} \left(2e\Delta\sqrt{p(1-p)}\right)^{j} \,\P\big[ \text{$X_{k}\in \mathbb{B}(v,j)$ for all $0\le k \le t$}  \big]\\
		&=&  \sum_{j=1}^{\infty} \left(2e\Delta\sqrt{p(1-p)}\right)^{j}  \sum_{t:j\le \Delta^{2t+1}}^{} \,\P\big[ \text{$X_{k}\in \mathbb{B}(v,j)$ $\forall$ $0\le k \le t$}  \big] \\
		&\le& \sum_{j=0}^{\infty} \left(2e\Delta\sqrt{p(1-p)}\right)^{j}  \,\E\big[ \mathtt{t}^{(v)}(j)  \big] = F(s_{p}),
	\end{eqnarray*}
	where $s_{p}=2e\Delta\sqrt{p(1-p)}$. Hence by Proposition \ref{prop:generating_ft}, $\E V^{(v)}<\infty$ whenever $|s_{p}|<(K_{\min})^{2}$. Observe that for $\Z^d$ with simple symmetric random walks, it is sufficient to have $p < (256 d^6 e^2)^{-1}$. On the oriented lattice $\vec{\mathbb{Z}}^{d}$, $ p< (64 d^6 e^2)^{-1}$ yields $\E_{p}V<\infty$.  
\end{proof}

\section*{acknowledgment}
The research of MD is supported by an NSF CAREER grant. JG was partially supported by the NSF grant DMS--1513340, Simons Foundation Award \#281309, and the Republic of Slovenia's Ministry of Science program P1--285. DS was partially supported by the NSF TRIPODS grant CCF--1740761.

\small{
	\bibliographystyle{amsalpha}  
	\bibliography{parking}  
}
	
\end{document}